\documentclass[a4paper, 10pt, twoside, notitlepage]{amsart}

\usepackage{amsmath,amscd}
\usepackage{amssymb}
\usepackage{amsthm}
\usepackage{comment}
\usepackage{graphicx, xcolor}

\usepackage{mathrsfs}
\usepackage[ocgcolorlinks, linkcolor=blue]{hyperref}

\usepackage{bm}
\usepackage{bbm}
\usepackage{url}

\usepackage[utf8]{inputenc}
\usepackage{mathtools,amssymb}
\usepackage{esint}
\usepackage{tikz}
\usepackage{dsfont}
\usepackage{relsize}
\usepackage{url}
\usepackage{xcolor}
\usepackage{graphicx}
\usepackage{mathrsfs}
\usepackage[shortlabels]{enumitem}
\usepackage{lineno}
\usepackage{amsmath}
\usepackage{enumitem}
\usepackage{amsthm} 
\usepackage{verbatim}
\usepackage{dsfont}
\numberwithin{equation}{section}

\allowdisplaybreaks

\mathtoolsset{showonlyrefs}

\graphicspath{{images/}}

\newtheorem{theorem}{Theorem}[section]
\newtheorem{lemma}[theorem]{Lemma}
\newtheorem{definition}[theorem]{Definition}
\newtheorem{corollary}[theorem]{Corollary}

\newtheorem{proposition}[theorem]{Proposition}

\newtheorem{remark}[theorem]{Remark}

\title[The nonlocal diffuse optical tomography problem]{Approximation and uniqueness results for the nonlocal diffuse optical tomography problem}

\author[Y.-H. Lin]{Yi-Hsuan Lin}
\address{Department of Applied Mathematics, National Yang Ming Chiao Tung University, Hsinchu, Taiwan}
\email{yihsuanlin3@gmail.com}

\author[P. Zimmermann]{Philipp Zimmermann}
\address{Departament de Matem\`atiques i Inform\`atica, Universitat de Barcelona, Barcelona, Spain}
\email{philipp.zimmermann@ub.edu}

\newcommand{\R}{{\mathbb R}}

\newcommand{\N}{{\mathbb N}}

\newcommand{\eps}{\epsilon}

\newcommand {\p} {\partial}

\newcommand{\LC}{\left(}
\newcommand{\RC}{\right)}
\newcommand{\wt}{\widetilde}

\newcommand{\schwartz}{\mathscr{S}}

\newcommand{\tempered}{\mathscr{S}^{\prime}}

\newcommand{\fourier}{\mathcal{F}}
\newcommand{\ifourier}{\mathcal{F}^{-1}}

\newcommand{\test}{\mathscr{D}}
\newcommand{\distr}{\mathscr{D}^{\prime}}

\newcommand{\abs}[1]{\left\lvert #1 \right\rvert}
\DeclareMathOperator{\Div}{div} 
\DeclareMathOperator{\supp}{supp} 

\begin{document}

	\maketitle
	\begin{abstract}
		We investigate the inverse problem of recovering the diffusion and absorption coefficients $(\sigma,q)$ in the nonlocal diffuse optical tomography equation 
        \begin{equation}
        \label{eq: eq abstract}
            (-\Div( \sigma \nabla))^s u+q u =0 \text{ in }\Omega
        \end{equation}
        from the nonlocal Dirichlet-to-Neumann (DN) map $\Lambda^s_{\sigma,q}$. The purpose of this article is to establish the following approximation and uniqueness results.
		\begin{enumerate}[(i)]
			\item\label{item 1 abstract}  \textit{Approximation:} We show that solutions to the conductivity equation
         \begin{equation}
         \label{eq: cond eq abstract}
          \Div( \sigma \nabla v)=0 \text{ in }\Omega
         \end{equation}
         can be approximated in $H^1(\Omega)$ by solutions to \eqref{eq: eq abstract} and the DN map $\Lambda_\sigma$ related to \eqref{eq: cond eq abstract} can be approximated by the nonlocal DN map $\Lambda_{\sigma,q}^s$.
			\item\label{item 2 abstract}  \textit{Local uniqueness:} We prove that the absorption coefficient $q$ can be determined in a neighborhood $\mathcal{N}$ of the boundary $\partial\Omega$ provided $\sigma$ is already known in $\mathcal{N}$.
            \item\label{item 3 abstract} \textit{Global uniqueness:} Under the same assumptions as in \ref{item 2 abstract}, and if one of the potentials vanishes in $\Omega$, then one can turn with the help of \ref{item 1 abstract} the local determination into a global uniqueness result.
		\end{enumerate}
        It is worth mentioning that the approximation result \ref{item 1 abstract} relies on the Caffarelli--Silvestre type extension technique and the geometric form of the Hahn--Banach theorem.
		
		\medskip
		
		

	\end{abstract}

	\tableofcontents

	\section{Introduction}\label{sec: introduction}
	
	In recent years, the study of nonlocal inverse problems attracted interest by many researchers. The first work in this field \cite{GSU20} concerned the unique determination of bounded potentials in the \emph{fractional Schr\"odinger equation}
	\begin{equation}\label{eq: frac schroedinger}
		((-\Delta)^s+q)u=0\text{ in }\Omega
	\end{equation}
	from the related (partial) Dirichlet-to-Neumann (DN) map. Here $\Omega\subset\R^n$ is a bounded domain and $0<s<1$. The proof of this uniqueness result relied on two essential ingredients, namely the \emph{unique continuation property} (UCP) of the fractional Laplacian $(-\Delta)^s$ and a \emph{Runge approximation}, which are special cases of the Propositions \ref{prop: UCP} and \ref{prop:(Runge-approximation-property)} and the Runge approximation allows to approximate any function in $L^2(\Omega)$ by solutions to \eqref{eq: frac schroedinger}. 

	By generalizations of these two beautiful and remarkable results, different articles could solve several inverse problems that are open in the corresponding local case or there even exist counterexamples to uniqueness. For example, one can uniquely determine in the nonlocal setting singular potentials or (linear) lower order local perturbations and furthermore the above approach has been extended to the case of higher order nonlocal operators or to detecting nonlinearities instead of potentials by combining it with other techniques as linearization and monotonicity methods (see \cite{bhattacharyya2021inverse,CLL2017simultaneously,CMR20,CMRU20,GLX,CLL2017simultaneously,cekic2020calderon,feizmohammadi2021fractional,harrach2017nonlocal-monotonicity,harrach2020monotonicity,GRSU18,GU2021calder,harrach2017nonlocal-monotonicity,harrach2020monotonicity,KRZ-2023,lin2020monotonicity,LL2020inverse,LL2022inverse,LLR2019calder,LLU2022para,KLW2021calder,RS17,ruland2018exponential,RZ2022unboundedFracCald,GU2021calder}). 
	
	Let us emphasize that the recovery of leading order coefficients for nonlocal operators has also been studied. These can be seen as the nonlocal counterparts to classical Calder\'on problem \cite{calderon} or the $p\,$-Calder\'on problem \cite{Salo:Zhong:2012}. In the works \cite{choulli2023fractional,CGRU2023reduction,feizmohammadi2021fractional_closed,feizmohammadi2021fractional,GU2021calder,ruland2023revisiting,LLU2022para,LLU2023calder,lin2023determining,RZ2022FracCondCounter,CRZ2022global,CRTZ-2022,RZ-low-2022,Frac-p-Lap-KLZ,LRZ2022calder,FKU24}, the authors investigated these types of inverse problems by utilizing either the DN map or the source-to-solution map as measurement operators.

	The simultaneous recovery of several coefficients in nonlocal partial differential equations (PDEs) from the related DN map is usually more involved, but there are still a few positive results into this direction. For example, the following results have been obtained:
	\begin{enumerate}[(i)]
		\item\label{item drift} In \cite{cekic2020calderon} the authors showed that the drift term $b$ and the potential $q$ in 
		\[
		( (-\Delta)^s+b\cdot\nabla +q) u=0\text{ in }\Omega
		\]
		can be recovered uniquely form the DN map.
		\item In \cite{LZ2023unique} a unique determination result has been obtained for all coefficients $(\rho,q)$ and kernel $K$ in the nonlocal porous medium equation
		\[
		\rho\partial_t u+L_K( u^m)+qu=0\text{ in }\Omega\times (0,T),
		\]
		where $m>1$ and $L_K$ is an \emph{elliptic integro-differential operator} of order $2s$, that is $L_K$ is given by 
		\begin{equation}
			\label{eq: integro-differential operator}
			L_Ku(x)=\text{p.v.}\int_{\R^n}K(x,y)(u(x)-u(y))\,dy
		\end{equation}
		with 
		\begin{equation}
			\label{eq: uniform ellipticity condition}
			K(x,y)=K(y,x)\quad\text{and}\quad \frac{\lambda}{|x-y|^{n+2s}}\leq K(x,y)\leq \frac{\lambda^{-1}}{|x-y|^{n+2s}}
		\end{equation}
		for some $\lambda>0$.
		\item\label{item nonlocal opt tomography} In \cite{zimmermann2023inverse} it is shown that the diffusion and absorption coefficients $(\gamma,q)$ in the \emph{nonlocal optical tomography equation}
		\begin{equation}
			\label{eq: nonlocal opt tom eq}
			L_{\gamma}^su+qu=0\text{ in }\Omega
		\end{equation}
		can be uniquely recovered from the DN map. Here, $L_{\gamma}^s$ is the integro-differential operator with kernel
		\begin{equation}
			\label{eq: kernel fractional cond op}
			K(x,y)=C_{n,s}\frac{\gamma^{1/2}(x)\gamma^{1/2}(y)}{|x-y|^{n+2s}}
		\end{equation}
		(cf.~\eqref{eq: integro-differential operator}),
		where $C_{n,s}>0$ is the usual normalization constant in the definition of the fractional Laplacian $(-\Delta)^s$ and $\gamma\colon \R^n\to\R$ is a uniformly elliptic function. It is noteworthy that in the endpoint case $s=1$ the operator $L_{\gamma}^s$ becomes the usual conductivity operator $-\Div(\gamma\nabla\cdot)$.
	\end{enumerate}
	Analogous uniqueness statements in the endpoint cases $s=1$ of \ref{item drift} and \ref{item nonlocal opt tomography} are generally not true. In this article, we consider a similar PDE as the nonlocal optical tomography equation \eqref{eq: nonlocal opt tom eq}, but where the operator $L_{\gamma}^s$ is replaced by a fractional power of the conductivity operator $-\Div(\sigma\nabla\cdot)$. For another interesting simultaneous determination result, we refer to the recent article \cite{FKU24}, where the geometric information and the potential on closed Riemannian manifolds are determined by using the local source-to-solution map.

	\subsection{Diffuse optical tomography problems and main results}
	
	In this section, we describe the model considered in this article, discuss its local counterpart and present our main results.
	
	\subsubsection{The nonlocal diffuse optical tomography problem}
	
	Let $\Omega \subset \R^n$ be a bounded domain with smooth boundary $\p \Omega$. Throughout this work, we suppose that the \emph{diffusion coefficient} $\sigma \in C^\infty(\R^n)$ is uniformly elliptic, that is
	\begin{equation}\label{ellipticity}
		\lambda \leq  \sigma(x)\leq \lambda^{-1}
	\end{equation}
	for some $\lambda\in (0,1)$, satisfies 
	\begin{equation}
		\label{eq: unity cond}
		\sigma(x)=1\text{ for }x\in \Omega_e,
	\end{equation}
	and the \emph{absorption coefficient} or \emph{potential} $q\in  L^\infty(\Omega)$ is nonnegative. Given such data, we consider the Dirichlet problem for the \emph{nonlocal diffuse optical tomography equation}
	\begin{equation}\label{equ: main}
		\begin{cases}
			(-\Div (\sigma \nabla))^s u + q u =0 &\text{ in }\Omega, \\
			u=f  & \text{ in }\Omega_e, 
		\end{cases}
	\end{equation}
	where 
	$$
	\Omega_e:=\R^n\setminus \overline{\Omega}
	$$ 
	denotes the exterior of $\Omega$ and $0<s<1$. Here $(-\Div (\sigma\nabla))^s$ is an elliptic integro-differential operator, which is rigorously defined in Section \ref{sec: preliminaries}. By standard methods, one sees that the Dirichlet problem \eqref{equ: main} is well-posed in the energy space $H^s(\R^n)$. 
	Let $W\subset \Omega_e$ be a nonempty bounded Lipschitz domain, then we can define the (partial) DN map 
	\begin{equation}
		\label{eq: DN map intro}
		\Lambda_{\sigma,q}^s: \wt H^s(W) \to H^{-s}(W), \quad f\mapsto \left. (-\Div( \sigma\nabla))^s u_f \right|_{W} ,
	\end{equation}
	where $u_f \in H^s(\R^n)$ is the unique solution to \eqref{equ: main}. With these results at hand, one can formulate the following inverse problem:
	

	\begin{enumerate}[\textbf{(IP)}]
		\item\label{Q:IP} \textbf{Inverse Problem.}  Can one uniquely determine the diffusion and absorption coefficients $(\sigma,q)$ in $\Omega$ from the partial DN map $\Lambda^s_{\sigma,q}$ given by  \eqref{eq: DN map intro}?
	\end{enumerate}
	
	\subsubsection{The classical diffuse optical tomography problem}
	
	Before presenting our main results, let us review the situation for the local counterpart of our model, which corresponds formally to the limiting case $s\to 1$. As explained below, this local inverse problem cannot be solved uniquely, this means $\sigma$ and $q$ are generally not uniquely determined by the related DN map $\Lambda_{\sigma,q}$. To see this, let us consider the \emph{diffuse optical tomography equation}
	\begin{equation}\label{equ: local opti}
		\begin{cases}
			-\Div  (\sigma \nabla w) + q w =0 &\text{ in }\Omega, \\
			w=g&\text{ on }\p \Omega,
		\end{cases}
	\end{equation}
	where $(\sigma,q)$ are again the diffusion and absorption coefficients. As in the nonlocal case, when $q\geq 0$ in $\Omega$, the well-posedness of \eqref{equ: local opti} guarantees the existence of the (full) DN map, which can be characterized by 
	\begin{equation}
		\Lambda_{\sigma,q}:H^{1/2}(\p \Omega)\to H^{-1/2}(\p \Omega), \quad g\mapsto  \left.\sigma \nabla w_g\cdot \nu \right|_{\p \Omega},
	\end{equation}
	where $w_g\in H^1(\Omega)$ is the unique solution to \eqref{equ: local opti}. Also here the final goal is to recover the coefficients $(\sigma,q)$ from the DN data $\Lambda_{\sigma,q}$ and we refer to this question as the \emph{diffuse optical tomography problem}. Thus, \ref{Q:IP} can be seen as the nonlocal analogon of this problem. 
	
	Now, one may observe that under suitable smoothness assumptions on $\sigma$ the classical \emph{Liouville transformation} $w\mapsto \sqrt{\sigma} w$, which is of fundamental importance in solving the classical Calder\'on problem \cite{sylvester1987global},
	maps every solution $w$ of the optical tomography equation
	\begin{equation}\label{R-equation}
		-\Div( \sigma \nabla w) + q w =0 \text{ in }\Omega
	\end{equation} 
	to a solution $\wt w=\sqrt{\sigma} w$ of the Schr\"odinger equation 
	\begin{equation}\label{S-equation}
		-\Delta \wt w + V \wt w=0 \text{ in }\Omega \quad  \text{with}\quad V=\frac{\Delta \sqrt{\sigma}}{\sqrt{\sigma}}+\frac{q}{\sigma}.
	\end{equation}
	The well-known counterexamples of Arridge and Lionheart \cite{AL1998nonuniqueness}, showing that this local inverse problem is not uniquely solvable, are indeed based on the above Liouville reduction. In fact, if $(\sigma_1,q_1)$ with $\sigma_1$ uniformly elliptic, $q_1\geq 0$ is given, then $(\sigma_2,q_2)$ with $\sigma_2=\sigma_0+\sigma_1$, $q_2=q_0+q_1$ has the same DN data, whenever $(\sigma_0,q_0)$ satisfies
	\begin{enumerate}[(i)]
		\item $\sigma_0\geq 0$ and $\sigma_0=0$ in a neighborhood of $\partial\Omega$,
		\item and the perturbation $q_0$ is given by
		\begin{equation}
			\label{eq: new potential}
			q_0=\sigma_2\left(\frac{\Delta \sqrt{\sigma_1}}{\sqrt{\sigma_1}}-\frac{\Delta\sqrt{\sigma_2}}{\sqrt{\sigma_2}}+\frac{q_1}{\sigma_1}\right)-q_1
		\end{equation}
	\end{enumerate}
	(see also \cite{zimmermann2023inverse}). So, even in the case $\sigma=1$ in a neighborhood of $\p \Omega$, one may lose uniqueness, and therefore, under these general conditions on $(\sigma,q)$ it seems impossible to uniquely determine $\sigma$ and $q$ simultaneously. For some positive results on the simultaneous recovery we refer the reader to \cite{harrach2009uniqueness,harrach2012simultaneous}. In these works it is shown that if $\sigma$ is piecewise constant and $q$ piecewise analytic, then uniqueness hold by applying the monotonicity method.
	
	Finally, let us mention that the above problem arises in steady state diffusion optical tomography, where light propagation is characterized by a diffusion approximation and the excitation frequency is set to zero. For a complete description of optical tomography including the derivation of \eqref{equ: local opti}, we refer the reader to the articles \cite{arridge1999optical} and  \cite{AL1998nonuniqueness}.

	\subsubsection{Approximation and uniqueness results for the nonlocal diffuse optical tomography problem}
	
	Our first main result is an approximation result for the DN map related to the classical conductivity equation, which reads.
	
	\begin{theorem}[Approximation]\label{thm: uniqueness of potential}
		Let $\Omega, W\subset \R^n$ be bounded domains with smooth boundaries such that $\overline{\Omega} \cap \overline{W}=\emptyset$. Suppose that the diffusion coefficient $\sigma\in C^\infty (\R^n)$ satisfies \eqref{ellipticity}--\eqref{eq: unity cond} and the absorption coefficient $ q\in L^{\infty}(\Omega)$ is nonnegative with compact support. Let $\Lambda^s_{\sigma,q}$ denote the DN map of
		\begin{equation}\label{equ: main j=12}
			\begin{cases}
				(-\Div (\sigma \nabla) )^s u + q u =0 &\text{ in }\Omega, \\
				u=f  & \text{ in }\Omega_e.
			\end{cases}
		\end{equation}
		For all $g\in H^{1/2}(\partial\Omega)$, there exists $( g_k)_{k\in\N}\subset C_c^{\infty}(W)$ such that 
		\begin{equation}
			\label{eq: approx bdry value}
			g=\lim_{k\to\infty}U_{g_k}|_{\partial\Omega}\text{ in }H^{1/2}(\partial\Omega)
		\end{equation}
		and
		\begin{equation}
			\label{eq: approximation prop DN}
			\Lambda_{\sigma}g=\lim_{k\to\infty}\partial_\nu U_{g_k}|_{\partial\Omega}\text{ in }H^{-1/2}(\partial\Omega).
		\end{equation}
		
		Here, $\Lambda_{\sigma}\colon H^{1/2}(\partial\Omega)\to H^{-1/2}(\partial\Omega)$ denotes the DN map related to the conductivity operator $\Div(\sigma\nabla\cdot)$ in $\Omega$ and for any $f\in C_c^{\infty}(W)$ the function $U_f$ is given by
		\[
		U_f(x)=\int_0^\infty y^{1-2s}\mathcal{P}^s_{\sigma}u_f (x,y) \, dy,
		\]
		where $u_f\in H^s(\R^n)$ is the unique solution to \eqref{equ: main j=12} and $\mathcal{P}^s_{\sigma}$ denotes the Caffarelli--Silvestre extension type operator with coefficient $\sigma$ (see Section~\ref{sec: extension problem}).
	\end{theorem}
	
	The previous approximation result relies on the following new Runge approximation:
	
	\begin{proposition}[New Runge approximation]\label{prop: density}
		Let $\Omega, W\subset \R^n$ be bounded domains with smooth boundaries such that $\overline{\Omega} \cap \overline{W}=\emptyset$. Suppose that the diffusion coefficient $\sigma\in C^\infty (\R^n)$ satisfies \eqref{ellipticity}--\eqref{eq: unity cond} and the absorption coefficient $ q\in L^{\infty}(\Omega)$ is nonnegative. Let 
		\begin{equation}
			\begin{split}
				\mathcal{D} & \vcentcolon= \left\{ U_f(x)=\int_0^\infty y^{1-2s}\mathcal{P}^s_{\sigma}u_f (x,y) \, dy \,: \, f \in C^\infty_c(W) \right\},\\
				\mathcal{D}' & \vcentcolon = \left\{\left. U_f \right|_{\p \Omega}: \, U_f \in \mathcal{D}\right\},
			\end{split}
		\end{equation}
		and 
		\begin{equation}\label{solution space}
			S \vcentcolon = \left\{ v \in H^1(\Omega): \, \Div (\sigma \nabla v) =0 \text{ in }\Omega  \right\}.
		\end{equation}
		Given $v\in S$, for any $\eps>0$, there exists $U_f \in \mathcal{D}$ such that 
		\begin{equation}\label{equ: Runge error}
			\left\|  U_f -v\right\|_{H^1(\Omega)} <\eps .
		\end{equation}
		Furthermore, given $g\in H^{1/2}(\p \Omega)$, for any $\eps>0$, one can find $\left. U_f \right|_{\p \Omega}\in \mathcal{D}'$ such that 
		\begin{equation}
			\label{eq: trace approx}
			\left\| \left. U_f \right|_{\p\Omega}-g \right\|_{H^{1/2}(\p \Omega)}<\eps.
		\end{equation}
	\end{proposition}
	
	Note that \eqref{eq: trace approx} is an immediate consequence of the solvability of the Dirichlet problem
	\begin{equation}
		\begin{cases}
			\Div(\sigma\nabla u)=0&\text{ in }\Omega,\\
			u=g&\text{ on }\partial\Omega
		\end{cases}
	\end{equation}
	for any $g\in H^{1/2}(\partial\Omega)$, \eqref{equ: Runge error} and the trace theorem.
	
	The above approximation result can be expected to hold since the Caffarelli-Silvestre (CS) type extension only brings the coefficient inside the nonlocal operator $(-\Div (\sigma\nabla))^s$ for the equation $( (-\Div (\sigma\nabla))^s +q ) u=0$ in $\Omega$ into the higher dimensional space. The potential only appears explicitly in the Robin-type boundary condition for the extension problem \eqref{eq: extension problem} on $\Omega\times \{0\}$. 	
	
	Furthermore, we found an interesting localization phenomena for the inverse problem of the nonlocal diffuse optical tomography equation, namely that one can determine $q$ in a certain neighborhood of the boundary, whenever $\sigma$ is known a priori in the same region without using the knowledge of $\sigma$ in the whole domain. 
	
	\begin{theorem}[Local uniqueness]\label{thm: uniqueness}
		Let $\Omega, W\subset \R^n$ be bounded domains with Lipschitz boundaries such that $\overline{\Omega} \cap \overline{W}=\emptyset$. Suppose that for $j=1,2$, the diffusion coefficient $\sigma_j\in C^\infty (\R^n)$ satisfies \eqref{ellipticity}--\eqref{eq: unity cond} and the absorption coefficient $ q_j\in C^0(\overline{\Omega})$ is nonnegative, for $j=1,2$.  If one has $\sigma_1=\sigma_2$ in a neighborhood $\mathcal{N}\subset \overline{\Omega}$ of $\p \Omega$, then \eqref{nonlocal DN map same} implies that $q_1=q_2$ in $\mathcal{N}$.	
	\end{theorem}

    It is noteworthy that for this result we are not using Theorem~\ref{thm: uniqueness of potential}. By unique continuation this clearly implies the following uniqueness result for real-analytic potentials.
	
	\begin{corollary}[Uniqueness of potential]
		\label{cor: uniqueness potential}
		Suppose that the assumptions of Theorem~\ref{thm: uniqueness} hold. If $q_j$ is real-analytic in domain $\Omega'\subset\Omega$ containing $\Omega\setminus \mathcal{N}$, then we have $q_1=q_2$ in $\Omega$.
	\end{corollary}

    Let us mention that if $\Lambda_{\sigma_1,q_1}^s=\Lambda^s_{\sigma_2,q_2}$ would imply $\Lambda_{\sigma_1,q_1}=\Lambda_{\sigma_2,q_2}$, then the counterexamples \eqref{eq: new potential} suggest that Theorem~\ref{thm: uniqueness} and Corollary~\ref{cor: uniqueness potential} hold. Combining Theorem~\ref{thm: uniqueness of potential} and Theorem~\ref{thm: uniqueness}, we obtain the following simultaneous unique determination result.
	
	\begin{theorem}[Unique determination]
		\label{thm: unique det}
		Let $\Omega, W\subset \R^n$ be bounded domains with Lipschitz boundaries such that $\overline{\Omega} \cap \overline{W}=\emptyset$. Suppose that for $j=1,2$, the diffusion coefficient $\sigma_j\in C^\infty (\R^n)$ satisfies \eqref{ellipticity}--\eqref{eq: unity cond} and the absorption coefficient $ q_j\in C^0(\overline{\Omega})$ is nonnegative, for $j=1,2$. 
        If one has $\sigma_1=\sigma_2$ in a neighborhood $\mathcal{N}\subset \overline{\Omega}$ of $\p \Omega$, either $q_1$ or $q_2$ vanishes in $\Omega$ and  
		\begin{equation}\label{nonlocal DN map same}
			\left.\Lambda^s_{\sigma_1,q_1}f\right|_W=\left.\Lambda^s_{\sigma_2,q_2}f\right|_W \text{ for all }f\in C_c^{\infty}(W),
		\end{equation}
		then there holds
		\begin{equation}
			\sigma_1=\sigma_2 \text{ and }q_1=q_2=0 \text{ in }\Omega.
		\end{equation}
	\end{theorem}

    Let us remark that the unique recovery of $\sigma_1=\sigma_2$ in $\Omega$ with $q_1=q_2=0$ has been studied in \cite{CGRU2023reduction}.  In this work, the authors investigated a new reduction method via the CS type extension, which leads to the observation that the local DN map $\Lambda_\sigma$ can be determined by the nonlocal DN map $\Lambda_{\sigma}^s=\Lambda_{\sigma,0}^s$. This idea is also of help in our nonlocal diffuse optical tomography problem. Conversely, when $\sigma$ is a given uniformly elliptic, matrix-valued function, the unique determination of the potential $q$ in \eqref{equ: main} has been studied in \cite{GLX}, which remains open for $n\geq 3$ for the local counterpart. Moreover, in \cite{zimmermann2023inverse} a global unique determination result for the nonlocal optical tomography problem (see above \ref{item nonlocal opt tomography}) has been found, but in contrast to the problem treated in this article there are several crucial differences:
	\begin{enumerate}[(a)]
		\item In analogy with the local case, the Liouville transformation $u\mapsto \gamma^{1/2}u$ maps the unique solution $u$ of the nonlocal optical tomography equation \eqref{eq: nonlocal opt tom eq} to the unique solution $v$ of the fractional Schr\"odinger equation
		\[
		( (-\Delta)^s+Q_{\gamma}) v=0\text{ in }\Omega\text{ with }Q_{\gamma}=-\frac{(-\Delta)^s( \sqrt{\gamma}-1)}{\sqrt{\gamma}}+\frac{q}{\gamma}.
		\]
		\item One does not need the assumption that the potential $q$ is compactly contained in $\Omega$.
		\item In \cite{zimmermann2023inverse}, it is shown that the assumptions on the potential and diffusion coefficients are sharp as otherwise one may construct counterexamples. 
	\end{enumerate}
	Let us note that in the problem in  \cite{zimmermann2023inverse} as well as in the nonlocal diffuse optical tomography problem \eqref{equ: main} with one potential vanishing, considered in this article, one can first recover the diffusion coefficient $\sigma$ or $\gamma$, respectively, and then the potential $q$.
	
	On the other hand, such a reduction could most likely not be used to establish Theorem~\ref{thm: unique det} by the ill-posedness of the local diffuse optical tomography problem. Although Theorem~\ref{thm: uniqueness} seems to hold in the classical case, the main obstruction is to expand the equality $q_1=q_2=0$ in $\mathcal{N}$ to the whole domain $\Omega$.


	\subsection{Ideas of the proof.}
	\label{subsec: ideas of proof} To prove Theorem~\ref{thm: uniqueness of potential}, we make use of the CS type extension for the nonlocal operator $( -\Div (\sigma \nabla))^s$ with $0<s<1$, that is 
	\begin{equation}\label{equ: extension problem}
		\begin{cases}
			\Div_{x,y}   (  y^{1-2s} \Sigma  \nabla_{x,y}\mathcal{P}_{\sigma}^s u )  =0 &\text{ in }\R^{n+1}_+,\\
			\mathcal{P}_{\sigma}^s u(x,0)=u(x) &\text{ on }\R^n,
		\end{cases}
	\end{equation}
	and it is well-known that there holds 
	\begin{equation}\label{extension Neumann}
		-\lim_{y\to 0}y^{1-2s}\p_y \mathcal{P}_{\sigma}^s u_f= d_s ( -\Div (\sigma \nabla) ) ^s u\text{ in }\R^n,
	\end{equation}
	where $\R^{n+1}_+:= \left\{(x,y) \in \R^{n+1} : \, x\in   \R^n \text{ and }y>0\right\}$, $d_s$ is a constant depending only on $s$, and $\Sigma$ is a $(n+1)\times (n+1)$ matrix of the form
	\begin{align}\label{Sigma(x)}
		\Sigma(x)=\left( \begin{matrix}
			\sigma(x) \mathbf{1}_n& 0\\
			0 & 1 \end{matrix} \right).
	\end{align}
	Here $\mathbf{1}_n$ denotes the $n\times n$ identity matrix. This type of extension problem was first studied in \cite{ST10} for the variable (matrix-valued) coefficients case and more general second order elliptic operators.

	Next, we consider the $A_2$-weighted $y$ integral of the CS type extension 
	\begin{equation}\label{the function v}
		U(x):=\int_0^\infty y^{1-2s}\mathcal{P}_{\sigma}^s u(x,y)\, dy.
	\end{equation}
	Integrating \eqref{equ: extension problem} with respect to the $y$-direction yields that 
	\begin{equation}\label{formal comp1}
		\begin{split}
			0&=\int_0^\infty 	\Div_{x,y}  (  y^{1-2s} \Sigma  \nabla_{x,y}\mathcal{P}_{\sigma}^s u )  \, dy \\
			&= \Div \left( \sigma \nabla \left( \int_0^\infty y^{1-2s}\mathcal{P}_{\sigma}^s u \, dy\right) \right)  +\int_0^\infty \p_y ( y^{1-2s} \p_y \mathcal{P}_{\sigma}^s u ) dy \\
			&= \Div (\sigma \nabla U) -\lim_{y\to 0}y^{1-2s}\p_y \mathcal{P}_{\sigma}^s u,
		\end{split}
	\end{equation} 
	which implies 
	\begin{equation}\label{formal comp2}
		-\Div ( \sigma \nabla U)= -\lim_{y\to 0}y^{1-2s}\p_y \mathcal{P}_{\sigma}^s u  =\underbrace{d_s ( -\Div(\sigma \nabla))^s u}_{\text{\eqref{extension Neumann}}}.
	\end{equation}
	Here we have assumed a suitable decay of $\mathcal{P}_{\sigma}^s u $ such that $\lim_{y\to \infty}y^{1-2s}\p_y \mathcal{P}_{\sigma}^s u =0$ at moment\footnote{This will be justified in Section \ref{sec: extension problem}.}.
	In particular, when $u_f$ is the solution to \eqref{equ: main}, the preceding derivations yield that 
	\begin{equation}\label{key: local-nonlocal}
		\Div ( \sigma \nabla U_f) = d_s q u_f \text{ in }\Omega,
	\end{equation}
	where $U_f$ is the function defined in \eqref{the function v} with $u=u_f$.

	Note that the left hand side of \eqref{key: local-nonlocal} is a local differential operator, and the right hand side comes from the nonlocal information of \eqref{equ: main}. Via \eqref{key: local-nonlocal}, one can see that the function $U_f$ may not solve the classical conductivity equation $\Div (\sigma \nabla v)=0$, since the source term $qu_f$ in \eqref{key: local-nonlocal} is in general nonzero. Surprisingly, we are able to prove that that any function in the set 
	$$
	\left\{ v\in H^1(\Omega): \, \Div (\sigma\nabla v)=0 \text{ in }\Omega\right\},
	$$
	can be approximated by a sequence of functions in 
	$$
	\left\{ U_f : \, f\in C^\infty_c(W) \right\}\subset H^1(\Omega)
	$$ 
	with respect to the $H^1(\Omega)$-norm (see Proposition~\ref{prop: density}). This result rests on the \emph{geometric form of the Hahn--Banach theorem} (see Section~\ref{subsec: new Runge} for detailed arguments).
	
	Furthermore, by this approximation property, the corresponding (local) DN map of \eqref{key: local-nonlocal} could be formally given by 
	\begin{align}
		\begin{split}
			\Lambda_{\sigma}:	H^{1/2}(\p \Omega)&\to H^{-1/2}(\p \Omega), \\
			\underbrace{\left.\int_0^\infty y^{1-2s}\mathcal{P}_{\sigma}^s u_f (x,y)\, dy\right|_{\p \Omega}}_{=\left. U_f\right|_{\p \Omega}} &\mapsto \underbrace{\left.\sigma \nabla \left(\int_0^\infty y^{1-2s}\mathcal{P}_{\sigma}^s u_f (x,y)\, dy \right) \cdot \nu \right|_{\p \Omega}}_{=\left. \sigma \nabla U_f\cdot \nu \right|_{\p \Omega}}.
		\end{split}
	\end{align}
	Thus, in this work, we aim to determine the local DN map $\Lambda_{\sigma}$ by the nonlocal one $\Lambda^s_{\sigma,q}$, even though the equation of $U_f$ involves more complicated elliptic equations, but the approximation property will help us to get rid of the zero order potential term. In short, we could determine $\sigma$ and $q$ separately, which involves both nonlocal and local results related to associated inverse problems.

	\subsection{Organization of the paper} 
	Our article is organized as follows. In Section \ref{sec: preliminaries}, we recall fractional Sobolev spaces and basic properties of the involved nonlocal operators. In Section \ref{sec: extension problem}, we introduce weighted Sobolev spaces and associated extension problem for our nonlocal operators. In Section \ref{sec: nonlocal to local}, we derive the relation between the nonlocal and the local problems. Moreover, we also provide the key approximation result in the same section. Using this Runge type approximation, we prove Theorem~\ref{thm: uniqueness of potential} in Section \ref{sec: proof of main thm}. Finally, in the same section, we also establish the proofs of Theorem~\ref{thm: uniqueness} and \ref{thm: unique det}.

	\section{The nonlocal problem}\label{sec: preliminaries}

	In this section, we review some known properties of the nonlocal operators considered in this article.
	
	\subsection{Fractional Sobolev spaces}	
	
	Let us start by recalling the definition of the fractional Sobolev spaces and the fractional Laplacian. 
	
	We denote by $\schwartz(\R^n)$ and $\tempered(\R^n)$ the space of Schwartz functions and tempered distributions, respectively. We define the Fourier transform $\fourier\colon \schwartz(\R^n)\to \schwartz(\R^n)$ by
	\begin{equation}
		\fourier f(\xi) \vcentcolon = \int_{\R^n} f(x)e^{-\mathrm{i}x \cdot \xi} \,dx,
	\end{equation}
	which is occasionally also denoted by $\widehat{f}$, and $\mathrm{i}=\sqrt{-1}$. By duality it can be extended to the space of tempered distributions and will again be denoted by $\fourier u = \widehat{u}$, where $u \in \tempered(\R^n)$, and we denote the inverse Fourier transform by $\ifourier$. 
	
	Given $s\in\R$, the $L^2$-based fractional Sobolev space $H^{s}(\R^{n})$ is the set of all tempered distributions $u\in\tempered(\R^n)$ such that
	\begin{equation}\notag
		\|u\|_{H^{s}(\mathbb{R}^{n})}\vcentcolon = \|\langle D\rangle^s u\|_{L^2(\R^n)}<\infty,
	\end{equation}
	where $\langle D\rangle^s$ is the Bessel potential operator of order $s$ having Fourier symbol $(1+|\xi|^2)^{s/2}$. 
	
	Next recall that the fractional Laplacian of order $s\geq 0$ is given as a Fourier multiplier via
	\begin{equation}\label{eq:fracLapFourDef}
		(-\Delta)^{s} u = \ifourier( \abs{\xi}^{2s}\widehat{u}(\xi)),
	\end{equation}
	for $u \in \tempered(\R^n)$ whenever the right-hand side is well-defined or as a singular integral by
	\begin{equation}
		\label{eq: sing int}
		(-\Delta)^su(x)=C_{n,s}\text{p.v.}\int_{\R^n}\frac{u(x)-u(y)}{|x-y|^{n+2s}}\,dy,
	\end{equation}
	where $C_{n,s}>0$ is a given constant depending only on $n,s$ and $\text{p.v.}$ denotes the Cauchy principal value.
	
	It is known that for $s\geq 0$ an equivalent norm on $H^s(\R^n)$ is given by
	\begin{equation}
		\label{eq: equivalent norm on Hs}
		\|u\|_{H^s(\R^n)}^*=( \|u\|_{L^{2}(\mathbb{R}^{n})}^{2}+\|(-\Delta)^{s/2}u\|_{L^{2}(\mathbb{R}^{n})}^{2})^{1/2}.
	\end{equation}
	Next we introduce some local variants of the above fractional Sobolev spaces. If $\Omega\subset \R^n$ is an open set and $s\in\mathbb{R}$,
	then we set
	\begin{align*}
		H^{s}(\Omega) & \vcentcolon =\left\{\,u|_{\Omega}\,: \, u\in H^{s}(\mathbb{R}^{n})\right\},\\
		\widetilde{H}^{s}(\Omega) & \vcentcolon =\text{closure of \ensuremath{C_{c}^{\infty}(\Omega)} in \ensuremath{H^{s}(\mathbb{R}^{n})}}.
	\end{align*}
	
	\subsection{Nonlocal elliptic operators}
	\label{subsec: nonlocal elliptic operators}
	
	Here, we recall the definition of the nonlocal operator $(-\Div(\sigma\nabla ))^s$, $0<s<1$, and discuss some of its properties. In this article we restrict our attention to the isotropic case, where $\sigma\colon\R^n\to \R$ is a smooth uniformly elliptic function such that $\sigma|_{\Omega_e}=1$. In this setting it is known that there exists a symmetric kernel $K^s_{\sigma}(x,y)$ comparable to the fractional Laplacian $(-\Delta)^s$, that is
	\begin{equation}
		\label{eq: ellipticity of op}
		\frac{c}{|x-y|^{n+2s}}\leq K^s_{\sigma}(x,y)\leq \frac{C}{|x-y|^{n+2s}},
	\end{equation}
	such that $(-\Div(\sigma\nabla ))^s$ defined via
	\begin{equation}\label{eq: weak form}
		\begin{split}
			\left\langle (-\Div(\sigma\nabla ))^su,v \right\rangle_{H^{-s}(\R^n)\times H^s(\R^n)} 
			=\int_{\R^{2n}}(u(x)-u(y))(v(x)-v(y))K^s_{\sigma}(x,y)\,dxdy,
		\end{split}
	\end{equation}
	for all $u,v\in H^{s}(\R^n)$ induces a bounded linear operator from $H^s(\R^n)$ to $H^{-s}(\R^n)$. Moreover, using the symmetry of $K^s_{\sigma}$ it is easy to see that the operator $(-\Div(\sigma\nabla ))^s$ is symmetric, that is one has
	\begin{equation}
		\label{eq: symmetry nonloc op}
		\left\langle (-\Div(\sigma\nabla ))^su,v \right\rangle_{H^{-s}(\R^n)\times H^s(\R^n)}=\left\langle (-\Div(\sigma\nabla ))^sv,u \right\rangle_{H^{-s}(\R^n)\times H^s(\R^n)}
	\end{equation}
	for all $u,v\in H^s(\R^n)$.
	
	By standard arguments, one easily shows that under the above assumptions on $\sigma$ and $0\leq q\in L^{\infty}(\Omega)$, the Dirichlet problem \eqref{equ: main} is well-posed for any $f\in H^s(\R^n)$, that is there exists a unique function $u\in H^s(\R^n)$ satisfying $u-f\in \widetilde{H}^s(\Omega)$ and 
	\begin{equation}
		\label{eq: bilinear form}
		B^s_{\sigma,q}(u,\varphi)\vcentcolon = \left\langle (-\Div(\sigma\nabla))^s u,\varphi \right\rangle_{H^{-s}(\R^n)\times H^s(\R^n)}+\int_{\Omega} qu\varphi\,dx=0
	\end{equation}
	for all $\varphi\in \widetilde{H}^s(\Omega)$. More concretely, one can use the kernel estimates \eqref{eq: ellipticity of op} and the fractional Poincar\'e inequality to deduce the boundedness and coercivity of the bilinear form \eqref{eq: bilinear form} on $\widetilde{H}^s(\Omega)$, then the Lax-Milgram theorem yields the well-posedness result as desired.
	Furthermore, if $f_1,f_2\in H^s(\R^n)$ satisfy $f_1-f_2\in\widetilde{H}^s(\Omega)$, then the corresponding unique solutions to \eqref{equ: main} coincide. Hence, the exterior value to solution map is well-defined on the abstract trace space $X_s=H^{s}(\R^n)/\widetilde{H}^s(\Omega)$ and for exterior value $f$ the unique solution will be denoted by $u_f$ in the rest of this article. Note that the well-posedness of \eqref{equ: main} also holds for general potential without sign assumption. Instead, one could consider suitable eigenvalue condition\footnote{For instance, the eigenvalue condition can be written as $0$ is not a Dirichlet eigenvalue of $( -\Div (\sigma\nabla))^s +q$ in $\Omega$.} to prove the well-posedness. However, in this work, we simply introduce the condition $q\geq 0$ in $\Omega$ for the self-consistency.
	
	\subsection{DN map}
	With the well-posedness of \eqref{equ: main} at hand, we can define the DN map rigorously by using the bilinear form \eqref{eq: bilinear form} induced by $(-\Div(\sigma\nabla ))^s+ q$ (see \eqref{eq: weak form}). For any $0<s<1$ and $(\sigma,q)$ as in Section~\ref{subsec: nonlocal elliptic operators}, we define the DN map $\Lambda_{\sigma,q}^s\colon X_s\to X_s^*$ via
	\begin{equation}
		\label{eq: DN map}
		\left\langle \Lambda_{\sigma,q}^s f,g \right\rangle=B_{\sigma,q}(  u_f, v_g) ,
	\end{equation}
	where $u_f\in H^s(\R^n)$ is the unique solution to \eqref{equ: main} and $v_g\in H^s(\R^n)$ is any representative of $g$. This immediately implies the following simple lemma:
	
	\begin{lemma}
		\label{lemma: diff DN maps}
		Let $\Omega,W\subset\R^n$ be bounded open sets such that $W\subset \Omega_e$ and $0<s<1$. Assume that $\sigma_1,\sigma_2\in C^{\infty}(\R^n)$ satisfy the conditions of Section~\ref{subsec: nonlocal elliptic operators} and $q_1,q_2\in L^{\infty}(\Omega)$ are given nonnegative potentials. Then there holds
		\begin{equation}\label{DN-integral id}
			\begin{split}
				\left\langle\Lambda_{\sigma_1,q_1}^s f_1,f_2 \right\rangle-\left\langle\Lambda_{\sigma_2,q_2}^s f_1,f_2\right\rangle=( B_{\sigma_1,q_1}-B_{\sigma_2,q_2}) ( u^{(1)}_{f_1},u^{(2)}_{f_2}) 
			\end{split}
		\end{equation}
		and
		\begin{equation}
			\label{DN-integral id 2}
			\left\langle\Lambda_{\sigma_1,q_1}^s f_1,f_2 \right\rangle-\left\langle\Lambda_{\sigma_2,q_2}^s f_1,f_2\right\rangle=B_{\sigma_1,q_1}( u^{(1)}_{f_1},u^{(2)}_{f_2})-B_{\sigma_2,q_2}( u^{(2)}_{f_1},u^{(2)}_{f_2})
		\end{equation}
		for all $f_1,f_2\in C_c^{\infty}(W)$, where $u_{f_k}^{(j)}\in H^s(\R^n)$ is the unique solution to \eqref{equ: main} with $\sigma=\sigma_j$, $q=q_j$ and exterior value $f_k$ for $1\leq j,k\leq 2$.
	\end{lemma}
	\begin{proof}
		First note that using \eqref{eq: DN map} and the symmetry of the bilinear form $B_{\sigma_j,q_j}$ we have
		\[
		\begin{split}
			\left\langle\Lambda_{\sigma_2,q_2}^s f_1,f_2\right\rangle&=B_{\sigma_2,q_2}( u_{f_1}^{(2)},u_{f_2}^{(2)})=B_{\sigma_2,q_2}( u_{f_2}^{(2)},u_{f_1}^{(2)})=\left\langle\Lambda_{\sigma_2,q_2}^s f_2,f_1\right\rangle
		\end{split}
		\]
		for all $f_1,f_2\in X_s$. 
		Again using \eqref{eq: DN map} and the symmetry of $B_{\sigma_2,q_2}$, we get
		\[
		\begin{split}
			\left\langle\Lambda_{\sigma_1,q_1}^s f_1,f_2 \right\rangle-\left\langle\Lambda_{\sigma_2,q_2}^s f_1,f_2 \right\rangle&=\left\langle\Lambda_{\sigma_1,q_1}^s f_1,f_2 \right\rangle-\left\langle\Lambda_{\sigma_2,q_2}^s f_2,f_1 \right\rangle\\
			&=B_{\sigma_1, q_1}( u_{f_1}^{(1)},u_{f_2}^{(2)})-B_{\sigma_2,q_2}( u_{f_2}^{(2)},u_{f_1}^{(1)})\\
			&=( B_{\sigma_1,q_1}-B_{\sigma_2,q_2}) ( u^{(1)}_{f_1},u^{(2)}_{f_2}).
		\end{split}
		\]
		This completes the proof of the identity \eqref{DN-integral id}. Next, note that \eqref{eq: DN map} implies
		\[
		\begin{split}
			\left\langle\Lambda_{\sigma_1,q_1}^s f_1,f_2\right\rangle &= B_{\sigma_1,q_1}( u^{(1)}_{f_1},u^{(2)}_{f_2}),\\
			\left\langle\Lambda_{\sigma_2,q_2}^s f_1,f_2\right\rangle &=B_{\sigma_2,q_2}( u^{(2)}_{f_1},u^{(2)}_{f_2}).
		\end{split}
		\]
		Thus, \eqref{DN-integral id 2} follows by subtracting them. 
	\end{proof}

	\subsection{Runge approximation}
	
	From \cite[Theorem 1.3]{GLX}, it is known that the Runge approximation holds for variable coefficients nonlocal elliptic operators. The proof of the Runge approximation is based on the unique continuation property (UCP in short) for the operator $( -\Div ( \sigma \nabla ))^s$. For readers' convenience, let us review these properties:
	
	\begin{proposition}[{UCP, \cite[Theorem 1.2]{GLX}}]\label{prop: UCP}
		Let $\sigma\in C^\infty(\R^n)$ satisfy \eqref{ellipticity}, $0<s<1$ and $\mathcal{O}\subset \R^n$ a nonempty open set. Suppose that $u\in H^s(\R^n)$ satisfies 
		\[
		u=( -\Div ( \sigma \nabla ))^s u=0\text{ in }\mathcal{O}.
		\]  
		Then $u\equiv0$ in $\R^n$.
	\end{proposition}
	
	\begin{proposition}[Runge approximation]
		\label{prop:(Runge-approximation-property)}
		Let $\Omega, W\subset\mathbb{R}^{n}$ be  bounded open sets with Lipschitz boundaries such that $\overline{\Omega}\cap \overline{W}=\emptyset$. Assume that $\sigma\in C^\infty(\R^n)$ satisfies \eqref{ellipticity} and $0\leq q\in L^{\infty}(\Omega)$. 
		Then the following inclusions are dense
		\begin{enumerate}[(i)]
			\item\label{Runge 1} $\left\{u_f|_{\Omega}: \, f\in C^\infty_c(W)\right\}\subset L^2(\Omega)$,
			\item\label{Runge 2} $\left\{u_f-f: \, f\in C^\infty_c(W)\right\}\subset \widetilde{H}^s(\Omega)$,
		\end{enumerate}
		where $u_f \in H^s(\R^n)$ is the solution to \eqref{equ: main}.
	\end{proposition}
	\begin{proof}
		The assertion \ref{Runge 1} has been proved in \cite[Lemma~5.6]{GLX} and \ref{Runge 2} follows by combining Proposition~\ref{prop: UCP} with \cite[Theorem~4.3]{RZ2022unboundedFracCald}.
	\end{proof}
	
	
	\section{The extension problem}\label{sec: extension problem}

	We review some important properties of the extension problem for $(-\Div (\sigma\nabla))^s$. We will also revisit some useful decay estimates for the solution of the extension problem. 
	Let us start by stating the following Caffarell--Silvestre (CS) type extension:
	\begin{definition}[CS-type extension]
		\label{def: CS type extension}
		Let $\sigma\in C^{\infty}(\R^n)$ satisfy the conditions in Section~\ref{subsec: nonlocal elliptic operators} and $0<s<1$. Suppose $p_t(x,z)$, $t>0$, is the (symmetric) heat kernel related to the parabolic equation
		\[
		\partial_t u-\Div (\sigma\nabla u)=0\text{ in }\R^n,\quad \lim_{t\downarrow 0}u=\delta_x.
		\]
		For any $u\in H^s(\R^n)$ we define its \emph{Caffarelli--Silvestre type extension} by
		\begin{equation}\label{CS extension of function}
			\mathcal{P}^s_{\sigma}u(x,y)\vcentcolon = \int_{\R^n}P_y(x,z)u(z)\,dz,
		\end{equation}
		where $P_y$ is the fractional Poisson kernel given by
		\begin{equation}
			\label{eq: frac poisson}
			P_y(x,z)=c_sy^{2s}\int_0^{\infty}e^{-y^2/4t}p_t(x,z)\frac{dt}{t^{1+s}}
		\end{equation}
		for some given constant $c_s>0$.
	\end{definition}
	
	\begin{remark}
		For the existence of a symmetric heat kernel $p_t$ as in the previous definition we refer to the monograph \cite{HeatKernelAnalysisManifolds} and furthermore by \cite{HeatKernelsSpectralTheory} we know that there are constants $\alpha_j,c_j>0$, $j=1,2$, such that
		\begin{equation}
			\label{eq: heat kernel comparable to classical}
			c_1e^{-\alpha_1|x-z|^2/4t}t^{-n/2}\leq p_t(x,z)\leq 
			c_2e^{-\alpha_2|x-z|^2/4t}t^{-n/2}
		\end{equation}
		for all $x,z\in\R^n$. Observe that without loss of generality we can assume that $\alpha_1\geq 1$ and $0<\alpha_2\leq 1$.
	\end{remark}

	\begin{lemma}\label{Lemma: degenerate and decay}
		Suppose that the assumptions of Definition~\ref{def: CS type extension} hold and denote by
		\begin{equation}
			\label{eq: poisson kernel frac lap}
			\widetilde{P}^s_y(x)=C_s\frac{y^{2s}}{( |x|^2+y^2)^{\frac{n+2s}{2}}}
		\end{equation}
		the fractional Poisson kernel related to the fractional Laplacian $(-\Delta)^s$.
		\begin{enumerate}[(i)]
			\item\label{equivalence of poisson kernels} There exist constants $C_1,C_2>0$ such that
			\begin{equation}
				\label{eq: comparison to poisson kernel of frac lap}
				C_1\widetilde{P}^s_{y/\sqrt{\alpha_1}}(x-z)\leq P_y(x,z)\leq C_2\widetilde{P}^s_{y/\sqrt{\alpha_2}}(x-z).
			\end{equation}
			\item\label{Lq estimate} For any $u\in L^p(\R^n)$ with $1\leq p\leq \infty$ and $1+\frac{1}{r}=\frac{1}{p}+\frac{1}{q}$, there holds
			\begin{equation}
				\label{eq: Lq estimate}
				\left\|\mathcal{P}_{\sigma}^su(\cdot,y)\right\|_{L^r(\R^n)}\lesssim y^{n(1- q)/q}\|u\|_{L^p(\R^n)},
			\end{equation}
			for any fixed $y>0$. Moreover, $\mathcal{P}_{\sigma}^s u \in L^p_{\mathrm{loc}}(\overline{\R^{n+1}_+},y^{1-2s})$ satisfies
			\begin{equation}
				\label{eq: lp loc estimate}
				\left\|\mathcal{P}_{\sigma}^s u \right\|_{L^p(\R^n\times (0,R),y^{1-2s})}\lesssim R^{2(1-s)/p}\|u\|_{L^p(\R^n)},
			\end{equation}
			for any $R>0$.
			\item\label{gradient estimate} For any $u\in L^p(\R^n)$ with $1\leq p\leq \infty$ and $1+\frac{1}{r}=\frac{1}{p}+\frac{1}{q}$, there holds
			\begin{equation}
				\label{eq: gradient bound}
				\left\|\nabla_{x,y}\mathcal{P}^s_{\sigma} u (\cdot,y) \right\|_{L^r(\R^n)}\lesssim y^{n(1-q)/q-1}\|u\|_{L^p(\R^n)}
			\end{equation}
			for any fixed $y>0$, where the exponent on the right hand side has to be interpreted as $-(n+1)$ in the case $p=1$.
			\item\label{solution} For any $u\in L^p(\R^n)$ with $1\leq p<\infty$, the function $\mathcal{P}_{\sigma}^s u$ solves the following extension problem
			\begin{equation}
				\label{eq: extension problem}
				\begin{cases}
					\Div_{x,y}   \left(  y^{1-2s} \Sigma  \nabla_{x,y} v \right)  =0 &\text{ in }\R^{n+1}_+,\\
					v (x,0)=u(x) &\text{ on }\R^n,
				\end{cases}
			\end{equation}
			where $\Sigma(x)$ is given by \eqref{Sigma(x)}.
			
			\item\label{neumann derivative} If $u\in H^s(\R^n)$, then one has $\nabla_{x,y}\mathcal{P}^s_{\sigma} u\in L^2(\R^{n+1}_+,y^{1-2s})$ satisfying
			\begin{equation}
				\label{eq: global regularity}
				\left\|\nabla_{x,y}\mathcal{P}^s_{\sigma} u\right\|_{L^2(\R^{n+1}_+,y^{1-2s})}\lesssim \|u\|_{H^s(\R^n)}
			\end{equation}
			and there holds
			\begin{equation}
				\label{eq: neumann derivative}
				-\lim_{y\to 0}y^{1-2s}\partial_y \mathcal{P}^s_{\sigma} u=d_s (-\Div(\sigma\nabla))^s u
			\end{equation}
			in $H^{-s}(\R^n)$ for some positive constant $d_s>0$.
		\end{enumerate}
	\end{lemma}

	\begin{proof}[Proof of Lemma \ref{Lemma: degenerate and decay}]
		We only provide the upper bound in \ref{equivalence of poisson kernels} as the lower bound works similarly. Using \eqref{eq: heat kernel comparable to classical} and change of variables we have
		\begin{equation}
			\label{eq: upper bound poisson kernel}
			\begin{split}
				P_y(x,z)&=c_s y^{2s}\int_0^{\infty}e^{-y^2/4t}p_t(x,z)\, \frac{dt}{t^{1+s}}\\
				&\leq Cy^{2s}\int_0^{\infty}e^{-\frac{y^2+\alpha_2|x-z|^2}{4t}}\, \frac{dt}{t^{n/2+1+s}}\\
				&= \underbrace{Cy^{2s}\int_0^{\infty}e^{-\frac{y^2+\alpha_2|x-z|^2}{4}\tau}\tau^{n/2+s-1}\, d\tau}_{\tau=1/t}\\
				&=\underbrace{C\frac{y^{2s}}{( y^2+\alpha_2|x-z|^2)^{\frac{n+2s}{2}}}\int_0^{\infty}e^{-\eta}\eta^{n/2+s-1}\, d\eta}_{\eta=\frac{y^2+\alpha_2|x-z|^2}{4}\tau}\\
				&=C\frac{y^{2s}}{( y^2+\alpha_2|x-z|^2)^{\frac{n+2s}{2}}}
			\end{split}
		\end{equation}
		for any $y>0$ and for some constant $C>0$ independent of the factors on both sides. In the above computation we used that the integral converges to the Gamma function $\Gamma(n/2+s)$. Taking $\alpha_2$ outside and rescaling gives the result. \\

		Next, we prove \ref{Lq estimate}. For this we may first observe that using \eqref{eq: comparison to poisson kernel of frac lap} we have
		\begin{equation}
			\label{eq: pointwise estimate}
			\left|\mathcal{P}_{\sigma}^s u\right|\leq \int_{\R^n}P_y(x,z)|u(z)|\,dz\leq C ( P^s_{y/\sqrt{\alpha_2}}\ast |u|) (x).
		\end{equation}
		Next, recall that by \cite[Lemma~7.1]{KRZ-2023} for any $s>0$ and $1\leq q<\infty$ there holds
		\begin{equation}
			\label{eq: lq estimate frac poisson}
			\left\|\widetilde{P}^s_y \right\|^q_{L^q(\R^n)}=C_s\frac{\omega_n}{2} y^{n(1-q)}B(n/2,n(q-1)/2+sq)
		\end{equation}
		for all $y>0$, where $\omega_n=\left|\partial B_1(0)\right|$ denotes the Lebesgue measure of $\p B_1(0)$,  and $B(x,y)$ is the Beta function\footnote{The constant $C_s>0$ is chosen in such a way that $\left\|\widetilde{P}^s_y \right\|_{L^1(\R^n)}=1$.}.
		Therefore, we can use Young's inequality and the previous estimate to obtain the bound \eqref{eq: Lq estimate}. The estimate \eqref{eq: lp loc estimate} follows from \eqref{eq: Lq estimate} and the same computation as in \cite[Lemma~7.2]{KRZ-2023}.\\

		Next, we turn to the proof of \ref{gradient estimate}. First, note that by \cite[Theorem 1.2]{CJKS20}, there holds
		\begin{equation}
			\label{eq: gradient estimate kernel}
			\left|\nabla_x p_t(x,z)\right|\leq c_3 t^{-\frac{n+1}{2}}e^{-\alpha_3 \frac{|x-z|^2}{4t}}
		\end{equation}
		for constants $c_3,\alpha_3>0$. Therefore, using Lebesgue's differentiation theorem and a similar computation as above we get
		
		\[
		\begin{split}
			\left|\nabla_x P_y(x,z)\right|&= c_s y^{2s}\left|\int_0^{\infty}e^{-y^2/4t}\nabla_x p_t(x,z)\, \frac{dt}{t^{1+s}} \right|\\
			&\leq C y^{2s}\int_0^{\infty}e^{-\frac{y^2+\alpha_3|x-z|^2}{4t}}\, \frac{dt}{t^{n/2+3/2+s}}\\
			&=\underbrace{C y^{2s}\int_0^{\infty}e^{-\frac{y^2+\alpha_3|x-z|^2}{4}\tau}\tau^{n/2+s-1/2}\, d\tau}_{\tau=1/t}\\
			&=\underbrace{C\frac{y^{2s}}{( y^2+\alpha_3|x-z|^2)^{\frac{n+2s+1}{2}}}\int_0^{\infty}e^{-\eta}\eta^{n/2+s}\, d\eta}_{\eta=\frac{y^2+\alpha_3|x-z|^2}{4}\tau}\\
			&= \frac{C}{y}\frac{y^{2(s+1/2)}}{( y^2+\alpha_3|x-z|^2)^{\frac{n+2(s+1/2)}{2}}}\\
			&=\frac{C}{y/\sqrt{\alpha_3}}\widetilde{P}^{s+1/2}_{y/\sqrt{\alpha_3}}(x-z).
		\end{split}
		\]
		Moreover, by Lebesgue's dominated convergence theorem we may calculate
		\[
		\begin{split}
			&\quad \, \partial_y P_y(x,z)\\
			&=c_s\partial_y \left( \int_0^{\infty}y^{2s}\,e^{-y^2/4t}p_t(x,z)\, \frac{dt}{t^{1+s}} \right) \\
			&= c_s\left(2s y^{2s-1}\int_0^{\infty}e^{-y^2/4t}p_t(x,z)\, \frac{dt}{t^{1+s}}-\frac{y^{2s+1}}{2}\int_{0}^{\infty}e^{-y^2/4t}p_t(x,z)\, \frac{dt}{t^{2+s}}\right).
		\end{split}
		\]
		Using \eqref{eq: heat kernel comparable to classical}, \eqref{eq: upper bound poisson kernel} and the same change of variables as above we get
		\[
		\begin{split}
			& \quad \, \left|\partial_y P_y(x,z)\right|\\
			&\leq C\left(y^{-1}P_y(x,z)+y^{2s+1}\int_0^{\infty}e^{-y^2/4t}p_t(x,z)\, \frac{dt}{t^{2+s}}\right)\\
			&\leq C\left(\frac{y^{2s-1}}{( y^2+\alpha_2|x-z|^2)^{\frac{n+2s}{2}}}+y^{2s+1}\int_0^{\infty}e^{-\frac{y^2+\alpha_2|x-z|^2}{4t}}\, \frac{dt}{t^{n/2+2+s}}\right)\\
			&=C\left(\frac{y^{2s-1}}{( y^2+\alpha_2|x-z|^2)^{\frac{n+2s}{2}}}+y^{2s+1}\int_0^{\infty}e^{-\frac{y^2+\alpha_2|x-z|^2}{4}\tau}\tau^{n/2+s}\, d\tau\right)\\
			&\leq C\left(\frac{y^{2s-1}}{( y^2+\alpha_2|x-z|^2)^{\frac{n+2s}{2}}}+\frac{y^{2s+1}}{( y^2+\alpha_2|x-z|^2)^{\frac{n+2(s+1)}{2}}}\right)\\
			&=\frac{C}{y}\left(\frac{y^{2s}}{( y^2+\alpha_2|x-z|^2)^{\frac{n+2s}{2}}}+\frac{y^{2(s+1)}}{( y^2+\alpha_2|x-z|^2)^{\frac{n+2(s+1)}{2}}}\right)\\
			&= \frac{C}{y/\sqrt{\alpha_2}}\left(\widetilde{P}^s_{y/\sqrt{\alpha_2}}(x-z)+\widetilde{P}_{y/\sqrt{\alpha_2}}^{s+1}(x-z)\right).
		\end{split}
		\]
		As in \eqref{eq: pointwise estimate} this implies
		\begin{equation}
			\label{eq: pointwise estimate gradient}
			\left|\nabla_x \mathcal{P}^s_{\sigma} u\right|\leq \frac{C}{y/\sqrt{\alpha_3}}\left(\widetilde{P}^{s+1/2}_{y/\sqrt{\alpha_3}}\ast |u|\right)
		\end{equation}
		and
		\begin{equation}
			\label{eq: pointwise estimate y derivative}
			\left|\partial_y \mathcal{P}^s_{\sigma} u \right|\leq \frac{C}{y/\sqrt{\alpha_2}}\left(\widetilde{P}^s_{y/\sqrt{\alpha_2}}+\widetilde{P}_{y/\sqrt{\alpha_2}}^{s+1}\right)\ast |u|.
		\end{equation}
		Therefore, by Young's inequality and \eqref{eq: lq estimate frac poisson} we may estimate
		\[
		\begin{split}
			&\quad \, \left\|\nabla_{x,y}\mathcal{P}^s_{\sigma} u (\cdot,y)\right\|_{L^r(\R^n)}\\
			&\leq \frac{C}{y}( \left\|\widetilde{P}^{s+1/2}_{y/\sqrt{\alpha_3}}\ast |u|\right\|_{L^r(\R^n)}+ \left\|\widetilde{P}^s_{y/\sqrt{\alpha_2}}\ast |u|\right\|_{L^r(\R^n)}+\left\|\widetilde{P}_{y/\sqrt{\alpha_2}}^{s+1}\ast |u|\right\|_{L^r(\R^n)}) \\
			&\leq Cy^{n(1-q)/q-1}\|u\|_{L^p(\R^n)},
		\end{split}
		\]
		which proves \ref{gradient estimate}. \\
		
		Now, we come to the proof of \ref{solution}. That $P^s_{y}(\cdot,z)\in L^2(\R^n)$ solves the PDE in \eqref{eq: extension problem} has been proven in \cite[Theorem~2.2]{ST10}. By standard result it then follows that $\mathcal{P}^s_{\sigma}(u)$ is a solution to the desired PDE. Furthermore, again by \cite[Theorem~2.2]{ST10} the boundary value is attained in the $L^p$ sense. \\

		Finally, we prove the assertion \ref{neumann derivative}. If $u\in C_c^{\infty}(\R^n)$, then we may conclude from \cite[Theorem~1.1]{ST10} that formula \eqref{eq: neumann derivative} holds in the $L^2$ sense. For the general case choose $u_k\in C_c^{\infty}(\R^n)$, $k\in\N$, such that $u_k\to u$ in $H^s(\R^n)$ as $k\to\infty$. Since $\mathcal{P}^s_{\sigma} u_k$ solves the extension problem \eqref{eq: extension problem}, using the previous $L^2$ convergence we get
		\[
		\begin{split}
			&\quad \, \int_{\R^{n+1}_+}y^{1-2s}\Sigma\nabla_{x,y}\mathcal{P}^s_{\sigma} u_k\cdot\nabla_{x,y}\mathcal{P}^s_{\sigma}u_k\,dxdy\\
			&=-\int_{\R^n\times \{y=0\}}\left(y^{1-2s}\partial_y \mathcal{P}^s_{\sigma} u_k\right) \mathcal{P}^s_{\sigma}u_k\,dx\\
			&=d_s\int_{\R^n}\left((-\Div(\sigma\nabla))^s u_k\right) u_k\,dx   \\
			&=d_s\left\langle (-\Div(\sigma\nabla))^s u_k,u_k \right\rangle_{H^{-s}(\R^n)\times H^s(\R^n)}.
		\end{split}
		\]
		Now, using the linearity of $u\mapsto \mathcal{P}^s_{\sigma} u$, \eqref{eq: gradient bound} with $r=p=2$ and Fatou's lemma we may estimate
		\[
		\begin{split}
			&\quad \, \int_{\R^{n+1}_+}y^{1-2s}\Sigma\nabla_{x,y}\mathcal{P}^s_{\sigma}u\cdot\nabla_{x,y}\mathcal{P}^s_{\sigma}u\,dxdy\\
			&\leq \liminf_{k\to\infty}\int_{\R^{n+1}_+}y^{1-2s}\Sigma\nabla_{x,y}\mathcal{P}^s_{\sigma}u_k\cdot\nabla_{x,y}\mathcal{P}^s_{\sigma}u_k\,dxdy\\
			&= d_s\liminf_{k\to\infty}\left\langle(-\Div(\sigma\nabla))^s u_k,u_k\right\rangle_{H^{-s}(\R^n)\times H^s(\R^n)}\\
			&=d_s\left\langle(-\Div(\sigma\nabla))^s u,u \right\rangle_{H^{-s}(\R^n)\times H^s(\R^n)}.
		\end{split}
		\]
		By the uniform ellipticity of $\sigma$, we directly get the desired estimate \eqref{eq: global regularity}. Finally, we can argue as in \cite[Proposition~4.3]{GLX} to conclude that \eqref{eq: neumann derivative} holds for general $u\in H^s(\R^n)$. This concludes the proof.
	\end{proof}

	\begin{remark}
		To make the formal computations \eqref{formal comp1} rigorous, we only need to check 
		\begin{equation}\label{decay as y to infty}
			\lim_{y\to \infty}y^{1-2s}\p_y v =0 \text{ in }\R^n,
		\end{equation} 
		where $v$ is the solution to the extension problem \eqref{equ: extension problem}. To this end, we can choose $r=\infty$, $p=q=2$ in \eqref{eq: gradient bound}, then the limit \eqref{decay as y to infty} holds automatically.
	\end{remark}

	\section{From nonlocal to local}\label{sec: nonlocal to local}
	
	In this section, we derive key ingredients in order to prove our main results. In Section~\ref{subsec: auxiliary lemmas} we prove a regularity result for $H^1_{\mathrm{loc}}$-solutions of
	\[
	-\Div(\sigma\nabla v)=d_s(-\Div (\sigma\nabla))^s u\text{ in }\R^n
	\]
	for $u\in H^s(\R^n)$ and review the local PDE solved by the functions $U_f$, which are defined via the equation \eqref{the function v} with $u=u_f$. Then in Section~\ref{subsec: new Runge} we provide via the geometric Hahn--Banach theorem our new Runge approximation, which allows to approximation solutions of $\Div(\sigma\nabla v)=0$ by the functions $U_f$.
	
	\subsection{Some auxiliary lemmas}
	\label{subsec: auxiliary lemmas}
	
	We begin with the aforementioned basic regularity result.
	
	\begin{lemma}[Regularity result]\label{lem: regularity}
		Let $\Omega\subset\R^n$ be a bounded domain, $0<s<1$, $\sigma\in C^{\infty}(\R^n)$ a uniformly elliptic diffusion coefficient with $\sigma|_{\Omega_e}=1$ and $u\in H^s(\R^n)$. Suppose that $v\in H^1_{\mathrm{loc}}(\R^n)$ is a solution to 
		\begin{equation}\label{equ: reg lemma}
			-\Div (\sigma \nabla v) =d_s(-\Div (\sigma\nabla))^s u \text{ in }\R^n,
		\end{equation}
		where $d_s$ is some constant depending only on $s\in (0,1)$. Then the Hessian $\nabla^2 v\in H^{-s}(\R^n;\R^{n \times n})$ satisfies 
		\begin{equation}
			\label{eq: estimate of hessian}
			\left\|\nabla^2 v\right\|_{H^{-s}(\R^n)}\leq C\|-\Delta v\|_{H^{-s}(\R^n)}\leq C( \|\nabla v\|_{L^2(\Omega)}+\|u\|_{H^s(\R^n)}) ,
		\end{equation}
		for some constant $C>0$ independent of $u$ and $v$, but depending on $\|\sigma\|_{C^{0,1}(\R^n)}$ . Moreover, we have $-\Div(\sigma\nabla v)\in H^{-s}(\R^n)$ 
		and equation \eqref{equ: reg lemma} holds in $H^{-s}(\R^n)$.
	\end{lemma}
	
	\begin{proof}
		By \eqref{equ: reg lemma}, let us first note that there holds
		\[
		-\sigma\Delta v=\nabla \sigma\cdot \nabla v+d_s(-\Div(\sigma\nabla))^su
		\]
		in $\distr(\R^n)$. Using that $(-\Div (\sigma\nabla))^s$ is a bounded linear operator from $H^s(\R^n)$ to $H^{-s}(\R^n)$ and $\sigma|_{\Omega_e}=1$, we deduce that there holds
		\[\begin{split}
			\left|\int_{\R^n}\nabla v\cdot\nabla (\sigma\varphi)\,dx\right|&\leq C( \|\nabla \sigma\|_{L^{\infty}(\Omega)}\|\nabla v\|_{L^2(\Omega)}\|\varphi\|_{L^2(\Omega)}+\|u\|_{H^s(\R^n)}\|\varphi\|_{H^s(\R^n)}) \\
			&\leq C( \|\nabla v\|_{L^2(\Omega)}\|\varphi\|_{L^2(\Omega)}+\|u\|_{H^s(\R^n)}\|\varphi\|_{H^s(\R^n)})
		\end{split}
		\]
		for all $\varphi\in C_c^{\infty}(\R^n)$. As $\sigma\in C^{\infty}(\R^n)$ is uniformly elliptic, we can replace $\varphi\in C_c^{\infty}(\R^n)$ by $\psi/\sigma\in C_c^{\infty}(\R^n)$ with $\psi\in C_c^{\infty}(\R^n)$ to get
		\[
		\begin{split}
			\left|\int_{\R^n}\nabla v\cdot\nabla \psi\,dx\right|&\leq C( \|\nabla v\|_{L^2(\Omega)}\|\psi\|_{L^2(\Omega)}+\|u\|_{H^s(\R^n)}\|\psi\|_{H^s(\R^n)}), 
		\end{split}
		\]
		for all $\psi \in C_c^{\infty}(\R^n)$. Here, we used the uniform ellipticity \eqref{ellipticity} of $\sigma$, $\sigma|_{\Omega_e}=1$ and the resulting Lipschitz continuity of $\sigma$ on $\R^n$, which implies that the multiplication map $H^s(\R^n)\ni v\mapsto \sigma^{-1} v\in H^s(\R^n)$ is bounded (see \cite[Lemma~5.3]{DNPV12}). More precisely, we used that the Lipschitz continuity and uniform ellipticity of $\sigma$ ensure that $\sigma^{-1}$ is Lipschitz continuous. This in turn follows from the simple estimate
		\[
		|\sigma^{-1}(x)-\sigma^{-1}(y)|=\left|\frac{\sigma(x)-\sigma(y)}{\sigma(x)\sigma(y)}\right|\leq C|\sigma(x)-\sigma(y)|\leq C[\sigma]_{C^{0,1}(\R^n)}|x-y|,
		\]
		where in the first inequality we used the uniform ellipticity of $\sigma$ and $[\sigma]_{C^{0,1}(\R^n)}$ denotes the usual Lipschitz seminorm.

		Thus, the distribution $-\Delta v$ satisfies
		\[
		|\langle -\Delta v,\varphi\rangle|\leq C( \|\nabla v\|_{L^2(\Omega)}+\|u\|_{H^s(\R^n)}) \|\varphi\|_{H^s(\R^n)}
		\]
		for all $\varphi\in C_c^{\infty}(\R^n)$ and hence continuously extends to an element in $H^{-s}(\R^n)$ with 
		\[
		\|-\Delta v\|_{H^{-s}(\R^n)}\leq C( \|\nabla v\|_{L^2(\Omega)}+\|u\|_{H^s(\R^n)}) .
		\]
		Note that
		\[
		\widehat{\partial_{jk}v}=\frac{\xi_j\xi_k}{|\xi|^2}\widehat{-\Delta v}
		\]
		for all $1\leq j,k\leq n$ and hence by Plancherel's theorem there holds
		\[
		\begin{split}
			\left\|\partial_{jk}v \right\|_{H^{-s}(\R^n)}&=C\left\|\langle \xi\rangle^{-s}\widehat{\partial_{jk}v}\right\|_{L^2(\R^n)}=C\left\|\langle \xi\rangle^{-s}\frac{\xi_j\xi_k}{|\xi|^2}\widehat{-\Delta v}\right\|_{L^2(\R^n)}\\
			&\leq C\left\|\langle \xi\rangle^{-s}\widehat{-\Delta v}\right\|_{L^2(\R^n)}=C\left\|-\Delta v\right\|_{H^{-s}(\R^n)}
		\end{split}
		\]
		for all $1\leq j,k\leq n$. Hence, we can conclude the proof of \eqref{eq: estimate of hessian}. Next, recall that we have
		\[
		-\Div(\sigma\nabla v)=-\sigma\Delta v-\chi_{\Omega}\nabla \sigma \cdot\nabla v
		\]
		in $\distr(\R^n)$, where we denote by $\chi_A$ its characteristic function for a set $A\subset \R^n$. Hence, we have
		\[
		\begin{split}
			\|-\Div(\sigma\nabla v)\|_{H^{-s}(\R^n)}&\leq \|-\sigma\Delta v\|_{H^{-s}(\R^n)}+\|\chi_{\Omega}\nabla \sigma \cdot\nabla v\|_{H^{-s}(\R^n)}\\
			&\leq C \|-\Delta v\|_{H^{-s}(\R^n)}+\|\nabla \sigma \cdot\nabla v\|_{L^2(\Omega)}\\
			&\leq C( \|\nabla v\|_{L^2(\Omega)}+\|u\|_{H^s(\R^n)}) .
		\end{split}
		\]
		Here we used again the Lipschitz continuity of $\sigma$ and \cite[Lemma~5.3]{DNPV12}. This shows that $-\Div(\sigma\nabla v)\in H^{-s}(\R^n)$. Finally, the fact that \eqref{equ: reg lemma} holds in $H^{-s}(\R^n)$ follows from the density of $C_c^{\infty}(\R^n)$ in $H^s(\R^n)$.

		Indeed, let $\varphi\in H^s(\R^n)$ and choose $\varphi_k\in C_c^{\infty}(\R^n)$ such that $\varphi_k\to \varphi$ in $H^s(\R^n)$. Then using $-\Div(\sigma\nabla v), (-\Div(\sigma\nabla))^su\in H^{-s}(\R^n)$, we obtain
		\[
		\begin{split}
			\left\langle -\Div(\sigma\nabla v),\varphi \right\rangle_{H^{-s}(\R^n)\times H^s(\R^n)}&=\lim_{k\to\infty}\langle -\Div(\sigma\nabla v),\varphi_k\rangle_{H^{-s}(\R^n)\times H^s(\R^n)}\\
			&=\lim_{k\to\infty}\int_{\R^n}\sigma\nabla v\cdot\nabla \varphi_k\,dx\\
			&=d_s\lim_{k\to\infty}\left\langle (-\Div(\sigma\nabla))^su,\varphi_k\right\rangle_{H^{-s}(\R^n)\times H^s(\R^n)}\\
			&=d_s\left\langle  (-\Div(\sigma\nabla))^su,\varphi\right\rangle_{H^{-s}(\R^n)\times H^s(\R^n)}.
		\end{split}
		\]
		This completes the proof.
	\end{proof}
	
	\begin{lemma}\label{Lem: C-S reduction}
		Let $\Omega\subset\R^n$ be a smoothly bounded domain and $0<s<1$. Assume that $\sigma$ is a smooth uniformly elliptic function with $\sigma|_{\Omega_e}=1$ and $0\leq q\in L^{\infty}(\Omega)$. Let $u_f\in H^s(\R^n)$ be the unique solution to \eqref{equ: main} with $f\in H^s(\R^n)$ and let $\mathcal{P}^s_{\sigma}u_f$ denote the Caffarelli-Silvestre type extension of $u_f$. Then
		\begin{equation}
			\label{eq: integrated version}
			U_f(x)\vcentcolon = \int_0^{\infty}y^{1-2s} \mathcal{P}^s_{\sigma}u_f(x,y) \,dy
		\end{equation}
		belongs to $H_{\mathrm{loc}}^1(\R^n)$. Moreover, the function $U_f$ solves 
		\begin{equation}\label{one key equ}
			-\Div (\sigma \nabla v ) = d_s (-\Div (\sigma \nabla))^s u_f \text{ in }\R^n
		\end{equation}
		and
		\begin{equation}\label{reduced equation}
			\begin{cases}
				\Div( \sigma \nabla v)=d_s q u_{f}  &\text{ in }\Omega, \\
				v =\left. U_f \right|_{\partial\Omega} &\text{ on }\p \Omega.
			\end{cases}
		\end{equation}
	\end{lemma}
	
	\begin{proof}

		By \cite[Proposition~6.1]{CGRU2023reduction}, it is known that $U_f\in H^1_{\mathrm{loc}}(\R^n)$ and by \cite[Theorem~3]{CGRU2023reduction} that it then solves \eqref{one key equ}. Now, by using the nonlocal equation \eqref{equ: main}, the equation \eqref{reduced equation} holds true as we want.	
	\end{proof}

	\begin{remark}
		Let $f_1, f_2\in C^\infty_c(W)$ be arbitrary exterior data, $u^{(j)}_{f_\ell}\in H^s(\R^n)$ be the solution to \eqref{equ: main j=12}, for $j=1,2$. Consider the Caffarelli-Silvestre type extension $\mathcal{P}^s_{\sigma_j}(u_{f_{\ell}}^{(j)})$
		for $j,\ell \in \{1,2\}$.
		By Lemma \ref{Lem: C-S reduction}, then the function
		\begin{equation}\label{the function v_j}
			U^{(j)}_{f_\ell}(x):=\int_0^\infty y^{1-2s}\mathcal{P}^s_{\sigma_j} u_{f_\ell}^{(j)}(x,y) \, dy
		\end{equation}
		solves the equation
		\begin{equation}\label{reduced local-nonlocal}
			-\Div (  \sigma_j \nabla U^{(j)}_{f_\ell} ) =d_s ( -\Div (  \sigma_j \nabla )  )^s u_{f_\ell}^{(j)} \text{ in }\R^n,
		\end{equation}
		for $j,\ell \in \{1,2\}$. Let us point out that the superscript $(j)$ corresponds to have conductivity $\sigma_j$ and potential $q_j$ in the extension problem and the nonlocal PDE, respectively, the subscript $f_{\ell}$ to the exterior data in the nonlocal PDE.
		In particular, as $j=\ell \in \{1,2\}$, $U^{(j)}_{f_j}$ satisfies
		\begin{equation}\label{reduced equation j=12}
			\begin{cases}
				\Div ( \sigma_j \nabla v)  =d_s q_j u_{f_j}^{(j)}  &\text{ in }\Omega, \\
				v =\left. U^{(j)}_{f_j} \right|_{\p \Omega} &\text{ on }\p \Omega
			\end{cases}
		\end{equation}
		in the $H^{-s}(\Omega)$ sense (see Lemma~\ref{lem: regularity}).
		
	\end{remark}

	\subsection{A new Runge approximation}
	\label{subsec: new Runge}

	With the preceding analysis, we want to prove our new Runge type approximation result (Proposition~\ref{prop: density}).
	
	We first give a formal proof for the case $s=1/2$ and $\sigma=1$.
	
	\begin{proof}[A formal proof of Proposition \ref{prop: density}]
		To demonstrate our idea, let us consider the case $s=1/2$ and $\sigma=1$. 
		Suppose \eqref{equ: Runge error} does not hold, i.e., there is a $v\in S$ and $\eps>0$ such that 
		\begin{equation}
			\label{eq: contradiction assumption}
			\left\| U_f - v \right\|_{H^1(\Omega)}\geq \eps >0, \text{ for all }f\in C^\infty_c(W).
		\end{equation}
		By the above inequality, $v\not\in \mathcal{D}$ and in particular $v\in S$ cannot be zero, otherwise it leads to a contradiction. Let 
		\begin{equation}
			A\vcentcolon=\overline{\mathcal{D}}^{H^1(\Omega)} \text{ and } B\vcentcolon= \{v\},
		\end{equation} 
		by the above condition, we have $A\cap B =\emptyset$, where $A$ and $B$ are closed convex sets in $H^1(\Omega)$ with $B$ being compact. In fact, for any $U\in A$, there exists $f\in C_c^{\infty}(W)$ such that 
		\[
		\|U-U_f\|_{H^1(\Omega)}<\eps/2.
		\]
		But then \eqref{eq: contradiction assumption} implies
		\[
		\|U-v\|_{H^1(\Omega)}\geq \|U_f-v\|_{H^1(\Omega)}-\|U-U_f\|_{H^1(\Omega)}\geq \eps/2>0.
		\]
		As this holds for any $U\in A$, we can deduce that $v\notin A$ and hence $A\cap B=\emptyset$.
		
		Now, by the geometric form of the Hahn-Banach theorem (for example, see \cite[Theorem 1.7]{Brezis}), there exists a (continuous) linear functional $\varphi\in (H^1(\Omega))^\ast=\wt H^{-1}(\Omega)$ (dual space of $H^1(\Omega)$) and $\alpha\in \R$ such that 
		\begin{equation}\label{equ: contradiction}
			\varphi ( U ) < \alpha < \varphi(v), \text{ for all }U\in A.
		\end{equation}
		Note that we necessarily have $\alpha>0$ as $f=0$ corresponds to $U_f=0$.
		The condition \eqref{equ: contradiction} especially implies that $\varphi (U_f)=0$ for all $f \in C^\infty_c(W)$. To see this assume that $\varphi( U_f) \neq 0$ for some $f\in C_c^{\infty}(W)$. As the mappings $f\mapsto U_f$ and $\varphi$ are linear, one can replace $f$ by $\mu f \in C^\infty_c(W) $ to obtain $\mu \varphi(U_f)\neq 0$ for all $0\neq \mu\in\R$. Hence, by choosing a suitable $\mu \in \R$ one gets a contradiction to \eqref{equ: contradiction}.
		
		Next, we aim to show that 
		\begin{equation}
			\varphi ( U_f ) =0 \text{ for all }f \in C^\infty_c(W) \implies \varphi(v)=0.
		\end{equation}
		To this end, let us consider the adjoint problem 
		\begin{equation}\label{adjoint s=1/2}
			\begin{cases}
				\Delta_{x,y} w = \varphi &\text{ in }\R^{n+1}_+ \\
				w(x,0)=0 &\text{ on } \Omega_e \times \{0\}, \\
				\displaystyle -\lim_{y\to0}\p_y w (x,y)+ q(x)w(x,0) =0  &\text{ on }\Omega \times \{0\}.
			\end{cases}
		\end{equation}
		For simplicity, let us set $\wt u_f (x,y) \vcentcolon=\mathcal{P}^{1/2}_{1}u_f (x,y)$ in the following derivation.
		By the equation \eqref{adjoint s=1/2}, direct computations show that 
		\begin{equation}
			\begin{split}
				\displaystyle0&=\varphi( U_f ) =\left\langle  \varphi ,  \int_0^\infty \wt u_f \,  dy\right\rangle_{\wt H^{-1}(\Omega)\times H^1(\Omega)}\\
				&= \int_{\R^{n+1}_+ }  (  \Delta_{x,y}w  ) \wt u_f \, dydx \\
				&= -\int_{\R^n} \wt u_f \lim_{y\to 0} \p_y w \, dx -\int_{\R^{n+1}_+} \nabla_{x,y}w \cdot \nabla_{x,y}\wt u_f \, dydx \\
				&=\underbrace{ -\int_{\Omega} \wt u_f \lim_{y\to 0} \p _yw \, dx}_{-\p_yw +qw=0 \text{ on }\Omega\times\{0\}}  - \int_{\Omega_e} \wt u_f \lim_{y\to 0}\p _yw \, dx + \int_{\R^n }w \lim_{y\to 0}\p_y \wt u_f \, dx \\
				&=-\int_{\Omega} u_f q w \, dx -\int_{W} f \lim_{y\to 0}\p_y w\, dx - \int_{\Omega}w (-\Delta)^{1/2} u_f \, dx \\
				&=-\int_{W} f \lim_{y\to 0}\p_y w\, dx ,
			\end{split}
		\end{equation}
		where we used $\LC  (-\Delta)^{1/2}+q \RC  u_f =0$ in $\Omega$ in the last equality, and recall that $\wt u_f$ satisfies
		\[
		\begin{cases}
			\Delta_{x,y}\widetilde{u}_f   =0 &\text{ in }\R^{n+1}_+,\\
			\widetilde{u}_f (x,0)=u_f(x) &\text{ on }\R^n
		\end{cases}
		\]
		with 
		\[
		-\lim_{y\to 0}\partial_y \widetilde{u}_f= (-\Delta)^{1/2} u_f\text{ in }\R^n.
		\]
		Here we used the fact that $d_{s}=1$ as $s=1/2$. Due to the arbitrariness of $f\in C^\infty_c(W)$, there must hold $\p_yw(x,0)=0$ in $W$. Moreover, the weak UCP for $\Delta_{x,y}$ and $\supp(\varphi)\subset \overline{\Omega}$ imply that $w=0$ in $\Omega_e\times (0,\infty)$. From this we infer that 
		\begin{equation}\label{w=p w =0}
		    w=0 \quad \text{and} \quad \p_\nu w =0 \text{ on } \p \Omega\times (0,\infty),
		\end{equation}
        where $\nu$ denotes the unit outwards pointing normal vector field of $\partial\Omega$. The first assertion of \eqref{w=p w =0} is clear, but let us give for the second assertion of \eqref{w=p w =0} with a more detailed argument. In the following we denote by $n$ the inwards pointing normal vector field on $\partial\Omega$ (i.e., $n=-\nu$ on $\p \Omega$). Then we define $\left.\partial_\nu w\right|_{\partial\Omega\times (0,\infty)}$ and $\left. \partial_n w\right|_{\partial\Omega\times (0,\infty)}$ weakly as
        \begin{equation}
        \label{eq: weak normal derivatives}
            \begin{split}
                \int_{\partial\Omega\times (0,\infty)}\partial_\nu w \rho\psi\,d\mathcal{H}^{n-1}&\vcentcolon = \left\langle \varphi,\int_0^{\infty}\rho \Psi\,dy\right\rangle_{\widetilde{H}^{-1}(\Omega)\times H^1(\Omega)}\\
                &\quad \, +\int_{\Omega\times (0,\infty)}\nabla_{x,y} w\cdot \nabla_{x,y} (\rho\Psi)\,dxdy,\\
                \int_{\partial\Omega\times (0,\infty)}\partial_n w \rho\psi\,d\mathcal{H}^{n-1}&\vcentcolon =\int_{\Omega_e\times (0,\infty)}\nabla_{x,y} w\cdot \nabla_{x,y} (\rho\Psi)\,dxdy,
            \end{split}
        \end{equation}
        where $\rho\in C_c^{\infty}((0,\infty))$, $\psi\in H^{1/2}(\partial\Omega)$ and $\Psi$ is in the first case any $H^1(\Omega)$ extension of $\psi$ and in the second case and $H^1(\Omega_e)$ extension of $\psi$ vanishing outside some sufficiently large compact set. These definitions correspond to the usual normal derivatives in the smooth case as can be seen by an integration by parts argument from the PDE \eqref{adjoint s=1/2}. Let us observe that the normal derivatives defined through \eqref{eq: weak normal derivatives} are independent of the extension $\Psi$ of $\psi\in H^{1/2}(\partial\Omega)$. 
        
        We demonstrate it for the first case. Suppose $\Psi,\Psi'\in H^1(\Omega)$ are two extensions of $\psi\in H^{1/2}(\partial\Omega)$ and $\rho\in C_c^{\infty}((0,\infty))$. Then the right hand side of the weak formulation of $\left.\partial_\nu w \right|_{\p \Omega\times (0,\infty)}$ coincide for $\Psi$ and $\Psi'$ if and only if 
        \[
            \left\langle \varphi,\int_0^{\infty}\rho (\Psi-\Psi')\,dy\right\rangle_{\widetilde{H}^{-1}(\Omega)\times H^1(\Omega)}+\int_{\R^{n+1}_+}\nabla_{x,y} w\cdot \nabla_{x,y} (\rho(\Psi-\Psi'))\,dxdy=0.
        \]
        As $w$ solves \eqref{adjoint s=1/2} and $\rho(\Psi-\Psi')\in H^1_{c,0}(\overline{\R^{n+1}_+})$, we see that the above identity holds (cf.~\eqref{eq: def of various function spaces} and \eqref{eq: weak formulation adjoint problem}) and so the normal derivative $\partial_\nu w$ is well-defined.
        Next, we show that
        \begin{equation}
        \label{eq: identity normal derivatives}
            \int_{\partial\Omega\times (0,\infty)}\partial_\nu w \rho\psi\,d\mathcal{H}^{n-1}=-\int_{\partial\Omega\times (0,\infty)}\partial_n w \rho\psi\,d\mathcal{H}^{n-1}
        \end{equation}
        for all $\rho\in C_c^{\infty}((0,\infty))$ and $\psi\in H^{1/2}(\partial\Omega)$. To see this fix some $\Psi\in H^1(\R^n)$ with $\Psi|_{\partial\Omega}=\psi$ in $H^{1/2}(\partial\Omega)$ and compact support. This can be done by the classical Sobolev trace and the extension theorem (cf.~\cite[Section~5.4]{EvansPDE}). Then we may compute
        \[
        \begin{split}
            &\quad \, \int_{\partial\Omega\times (0,\infty)}\partial_\nu w \rho\psi\,d\mathcal{H}^{n-1}\\
           & =\left\langle \varphi,\int_0^{\infty}\rho \Psi\,dy\right\rangle_{\widetilde{H}^{-1}(\Omega)\times H^1(\Omega)}+\int_{\R^{n+1}_+}\nabla_{x,y} w\cdot \nabla_{x,y} (\rho\Psi)\,dxdy\\
                &\quad \, - \int_{\Omega_e\times (0,\infty)}\nabla_{x,y} w\cdot \nabla_{x,y} (\rho\Psi)\,dxdy\\
                &=-\int_{\partial\Omega\times (0,\infty)}\partial_n w \rho\psi\,d\mathcal{H}^{n-1},
        \end{split}
        \]
        where the terms in the second line again vanished as $w$ solves \eqref{adjoint s=1/2} and $\rho\Psi\in H^1_{c,0}(\overline{\R^{n+1}_+})$. Now, by \eqref{eq: weak normal derivatives}, \eqref{eq: identity normal derivatives},  and as $w=0$ in $\Omega_e\times (0,\infty)$, we deduce that 
        \[
            \int_{\partial\Omega\times (0,\infty)}\partial_\nu w \rho\psi\,d\mathcal{H}^{n-1}=0
        \]
        for all $\rho\in C_c^{\infty}((0,\infty))$ and $\psi\in H^{1/2}(\partial\Omega)$. Hence, we have
        \[
        \partial_\nu w|_{\partial\Omega\times (0,\infty)}=0
        \]
        in the aforementioned weak sense as asserted.
        \\

		In the next step, we want to show that $\varphi(v)=0$. To get rid of boundary contributions on the set $\Omega\times \{0\}$, let us choose $\beta_1\in C^\infty_c((0,\infty))$ with $\supp ( \beta_1 ) \subset (1,2)$, $\beta_1 \geq 0$ and $\int_0^\infty\beta_1 \, dy=1$. Now, consider 
		$$
		\beta_k (y) \vcentcolon= \frac{1}{k}\beta_1(y/k) \text{ for } k\in \N.
		$$
		By a change of variables one can easily see that $\int_0^\infty \beta_k\, dy=1$ for all $k\in \N$.  Similarly as in the previous computations, one has  
		
		\begin{equation}
			\begin{split}
				-\varphi(v) &= - \varphi \left( v \int_0^\infty \beta_k \, dy  \right)   \\ 
				&=-\int_{\Omega\times (0,\infty)}(\Delta_{x,y}w)\beta_k v\,dydx\\
				&= -\underbrace{\int_{\p \Omega \times (0,\infty)} \p_\nu  w v \beta_k \, dyd\mathcal{H}^{n-1} }_{=0, \text{ since }\p_\nu w=0 \text{ on }\p \Omega\times(0,\infty)}+ \underbrace{\int_{\Omega\times \{0\}}\p_y w v \beta_k  \, dx }_{=0,\text{ since }\beta_k(0)=0}\\
				&\quad \, + \int_{\Omega\times (0,\infty)} \nabla_{x,y}w \cdot \nabla_{x,y}( v \beta_k ) dydx \\
				&= \int_{\Omega} \nabla v \cdot \nabla \left( \int_0^\infty w\beta_k \, dy\right) dx +\int_{\Omega\times (0,\infty)} v \p_y w \p_y\beta_k \, dydx \\
				&= \underbrace{\int_{\p \Omega} \p_\nu v \left( \int_0^\infty w\beta_k \, dy\right) d\mathcal{H}^{n-1}}_{=0, \text{ since }w=0 \text{ on }\p \Omega \times (0,\infty)} -\underbrace{\int_{\Omega} \Delta v\left( \int_0^\infty w\beta_k \, dy\right) dx}_{=0, \text{ since }\Delta v=0 \text{ in }\Omega} \\
				&\quad \,  +\int_{\Omega\times (0,\infty)} v \p_y w \p_y\beta_k \, dydx,
			\end{split}
		\end{equation}
		for all $k\in \N$. Hence, take $k\to \infty$, we have 
		\begin{equation}
			-\varphi(v)=\lim_{k\to \infty}\int_{\Omega\times (0,\infty)} v \p_y w \p_y\beta_k \, dydx.
		\end{equation}
		Our final aim is to prove that the above limit is zero. As claimed in \cite[Section 3.1]{CGRU2023reduction}, we can obtain 
		\begin{equation}
			\lim_{k\to \infty}\int_{\Omega\times (0,\infty)} v \p_y w \p_y\beta_k \, dydx= \lim_{k\to \infty}\left( k^{-2}\int_{\Omega\times (k,2k)} v \p_y w \p_y\beta_1 \, dydx \right) =0,
		\end{equation}
		which shows $\varphi(v)=0$ as desired. However, this contradicts the condition \eqref{equ: contradiction} so that $v\in S$ as asserted cannot exist. Finally, by the trace theorem, it is easy to see that any $H^{1/2}(\p \Omega)$ function can be approximated by a sequence of functions in $\mathcal{D}'$ as well. 
	\end{proof}
	
	The above derivation demonstrates the idea to show our new Runge approximation. Next, based on this approach, we want to prove the general case rigorously, that is in which $\sigma$ may not be 1 in $\Omega$ and $s\neq 1/2$. One can see from the formal proof that the adjoint equation \eqref{adjoint s=1/2} plays an essential role in this argument. Therefore, let us first study the existence of a solution to this problem. To this end, let us introduce the following subspaces of the homogeneous and non-homogeneous weighted Sobolev spaces:
	\begin{equation}
    \label{eq: def of various function spaces}
		\begin{split}
			\dot H^1_{0}(\R^{n+1}_+,y^{1-2s})&\vcentcolon = \bigg\{  g\in \dot H^1(\R^{n+1}_+,y^{1-2s}): \, g=0 \text{ on }\Omega_e \times \{0\} \bigg\},\\
			\dot H^1_{\mathrm{loc},0}(\overline{\R^{n+1}_+},y^{1-2s})&\vcentcolon = \bigg\{  g\in \dot H^1_{\mathrm{loc}}(\overline{\R^{n+1}_+},y^{1-2s}): \, g=0 \text{ on }\Omega_e \times \{0\} \bigg\},\\
			H^1_{c}(\overline{\R^{n+1}_+}, y^{1-2s})&\vcentcolon = \bigg\{  g\in H^1(\R^{n+1}_+,y^{1-2s}): \, g\text{ has compact support in }\overline{\R^{n+1}_+} \bigg\},\\
			H^1_{c,0}(\overline{\R^{n+1}_+}, y^{1-2s})&\vcentcolon = \bigg\{  g\in H^1_c(\overline{\R^{n+1}_+},y^{1-2s}): \, g=0 \text{ on }\Omega_e \times \{0\} \bigg\}.
		\end{split}
	\end{equation}
	
	\begin{lemma}[Solvability of the adjoint problem]\label{Lem: solvable}
		Let $\Omega\subset\R^n$ be a bounded Lipschitz domain and $0<s<1$. Assume that $\Sigma$ is given by \eqref{Sigma(x)} such that $\sigma\in C^\infty(\R^n)$ is uniformly elliptic with $\sigma=1$ in $\Omega_e$. Let $0\leq q\in L^\infty(\Omega)$, and $\varphi \in \wt H^{-1}(\Omega)$. Consider 
		\begin{equation}\label{adjoint s in 0,1}
			\begin{cases}
				\Div_{x,y} ( y^{1-2s}\Sigma \nabla_{x,y} w ) =y^{1-2s}\varphi  &\text{ in }\R^{n+1}_+, \\
				w=0  &\text{ on }\Omega_e \times \{0\}, \\
				\displaystyle -\lim_{y\to 0}y^{1-2s}\p_yw + d_s qw =0 &\text{ on }\Omega\times \{0\}.
			\end{cases}
		\end{equation}
		Then the problem \eqref{adjoint s in 0,1} is solvable in $\dot H^1_{\mathrm{loc},0}(\overline{\R^{n+1}_+}, y^{1-2s})$ in the sense that 
		\begin{equation}
			\label{eq: weak formulation adjoint problem}
			\begin{split}
				&\quad \, \int_{\R^{n+1}_+} y^{1-2s}\Sigma \nabla_{x,y}w \cdot \nabla_{x,y} \psi  \, dxdy+d_s \int_{\Omega\times \{0\}} q w \psi \, dx \\
				&= -\left\langle \varphi , \int_0^\infty y^{1-2s}\psi (\cdot,y)\, dy \right\rangle_{\wt H^{-1}(\Omega)\times H^1(\Omega)},
			\end{split}
		\end{equation}
		for any $\psi \in H^1_{c,0}(\overline{\R^{n+1}_+}, y^{1-2s})$.
	\end{lemma}
	
	\begin{remark}
		\label{remark: motivation notion of weak sol}
		Let us give a short explanation to require the identity \eqref{eq: weak formulation adjoint problem}. To this end regard $\varphi$ as a function supported in $\Omega$. Then multiplying \eqref{adjoint s in 0,1} by $\psi \in H^1_{c,0}(\overline{\R^{n+1}_+}, y^{1-2s})$ and integrating the resulting equation over $\R^{n+1}_+$ gives
		\[
		\begin{split}
			&\quad \, \int_{\Omega\times (0,\infty)}y^{1-2s}\varphi\psi\,dxdy\\
			&=\int_{\R^{n+1}_+}\Div_{x,y}( y^{1-2s}\Sigma \nabla_{x,y} w )\psi\,dydx\\
			&=-\int_{\R^n\times \{0\}}y^{1-2s}\partial_y w\psi\,dx-\int_{\R^{n+1}_+}y^{1-2s}\Sigma \nabla_{x,y}w\cdot\nabla_{x,y}\psi\,dydx\\
			&=-d_s\int_{\Omega\times \{0\}}qw\psi\,dx-\int_{\Omega_e\times \{0\}}y^{1-2s}\partial_y w\psi\,dx-\int_{\R^{n+1}_+}y^{1-2s}\Sigma \nabla_{x,y}w\cdot\nabla_{x,y}\psi\,dydx.
		\end{split}
		\]
		In the second equality we performed the usual integration by parts and in the third equality we used the boundary conditions. Now, by the assumption $\psi\in  H^1_{c,0}(\overline{\R^{n+1}_+}, y^{1-2s})$ the middle term in the last line vanishes and we arrive at the identity \eqref{eq: weak formulation adjoint problem}.
	\end{remark}
	
	\begin{remark}
		\label{Remark integration by parts}
		By the regularity for the function $w$ in Lemma \ref{Lem: solvable}, we could define the normal derivative $\lim_{y\to 0}y^{1-2s}\p_y w \in \dot H^{-s}_{\mathrm{loc}}(\R^n)$ as  
		\begin{equation}
			\label{eq: weak def of Neumann derivative}
			\begin{split}
				\displaystyle -\int_{\R^n}\psi(\cdot,0) \lim_{y\to 0}y^{1-2s}\p_y w\, dx & \vcentcolon=\int_{\R^{n+1}_+}y^{1-2s}\Sigma \nabla_{x,y}w\cdot \nabla_{x,y}\psi\, dxdy\\
				&\quad \, + \left\langle \varphi, \int_0^\infty y^{1-2s} \psi(\cdot,y)\, dy \right\rangle_{\wt H^{-1}(\Omega)\times H^1(\Omega)}
			\end{split}
		\end{equation} 
		for any $\psi \in H^1_{c}(\overline{\R^{n+1}_+}, y^{1-2s})$ (cf. Remark~\ref{remark: motivation notion of weak sol}).
	\end{remark}
	
	\begin{proof}[Proof of Lemma \ref{Lem: solvable}]
		The construction of solutions to \eqref{adjoint s in 0,1} is similar to the proof of \cite[Lemma 3.2]{CGRU2023reduction}. As $\varphi\in\widetilde{H}^{-1}(\Omega)$ has compact support, there exists a function $u_1\in \dot H^1(\R^n)$ solving
		\begin{equation}
			\Div ( \sigma\nabla u_1 ) = \varphi \text{ in }\R^n. 
		\end{equation}
		In fact, this solution can be obtained as a minimizer of the weakly lower semicontinuous, convex, coercive energy functional $E\colon \dot{H}^1(\R^n)\to\R$ defined by
		\begin{equation}\label{energy functional 1}
			\begin{split}
				E(\psi)\vcentcolon = \frac{1}{2}\int_{\R^n}\sigma |\nabla \psi|^2\,dx+\langle \varphi,\psi\rangle_{\widetilde{H}^{-1}(\Omega)\times H^1(\Omega)}
			\end{split}
		\end{equation}
		for $\psi\in \dot{H}^1(\R^n)$ (see~\cite[Chapter I, Theorem~1.2]{Variational-Methods}). The fact that $E$ is convex is immediate from the definition and hence let us shortly argue how one gets the other two properties. 
		\begin{enumerate}[(i)]
			\item\label{coercivity for E} (Coercivity). Note that by the Sobolev embedding, we have $\dot{H}^1(\R^n)\hookrightarrow L^{\frac{2n}{n-2}}(\R^n)$ and thus, by boundedness of $\Omega$, this ensures the continuity of the restriction $\dot{H}^1(\R^n)\ni \psi\mapsto \psi|_{\Omega}\in H^1(\Omega)$. This in turn guarantees the coercivity estimate
			\[
			E(\psi)\geq C_1\|\psi\|^2_{\dot{H}^1(\R^n)}-C_2\|\varphi\|_{\widetilde{H}^{-1}(\Omega)}^2 \geq -C_2\|\varphi\|_{\widetilde{H}^{-1}(\Omega)}^2
			\]
			for some constants $C_1,C_2>0$.
			
			\item (Weak lower semicontinuity). Note that weak lower semicontinuity follows from the fact that
			\begin{equation}
				\label{eq: equivalent norm easy case}
				\|\psi\|_{\sigma}\vcentcolon = \left(\int_{\R^n}\sigma|\nabla\psi|^2\,dx\right)^{1/2}
			\end{equation}
			is an equivalent norm on $\dot{H}^1(\R^n)$, as $\sigma$ is uniformly elliptic, and 
			\[
			\dot{H}^1(\R^n)\ni\psi\mapsto \langle \varphi,\psi|_{\Omega}\rangle_{\widetilde{H}^{-1}(\Omega)\times H^1(\Omega)}
			\] 
			is a continuous linear functional (see \ref{coercivity for E}).
		\end{enumerate}
		By $\wt u_1 (x,y)\equiv u_1(x)$, we denote the trivial extension of $u_1$ into the $y$-direction. \\

		Now, let us consider the problem
		\begin{equation}\label{adjoint u_2}
			\begin{cases}
				\Div_{x,y}( y^{1-2s}\Sigma \nabla_{x,y} \wt u_2) = 0 &\text{ in }\R^{n+1}_+, \\
				\wt u_2=0 &\text{ on }\Omega_e \times \{0\},\\
				\displaystyle-\lim_{y\to 0}y^{1-2s} \p_y \wt u_2 +d_sq \wt u_2 =-\lim_{y\to 0} y^{1-2s}\p_y \mathcal{P}^s_\sigma u_1   &\text{ on }\Omega\times \{0\},
			\end{cases}
		\end{equation}
		where $\mathcal{P}_\sigma^s$ is the Caffarelli-Silvestre extension operator given by \eqref{CS extension of function}. Performing a similar analysis as in the proof of \cite[Lemma 3.2]{CGRU2023reduction}, one sees that 
		\begin{equation}\label{P u_1 in H-s}
			\displaystyle \lim_{y \to 0} y^{1-2s}\p_y \mathcal{P}^s_{\sigma} u_1 \in \dot{H}^{-s}(\R^n)+L^2(\Omega)\hookrightarrow H^{-s}(\Omega).
		\end{equation}
		The relation \eqref{P u_1 in H-s} follows in fact by decomposing $u_1$ as
		\begin{equation}
			\label{eq: decomposition of u1}
			u_1=\eta_R u_1+(1-\eta_R)m_h(D)u_1+(1-\eta_R)m_{\ell}(D)u_1=v_R+v_h+v_{\ell},
		\end{equation}
		where $R>0$ is chosen such that $\overline{\Omega}\subset B_R$, $\eta_R\in C_c^{\infty}(B_{2R})$ satisfies $\eta_R|_{\overline{B}_R}=1$ and the Fourier multipliers $m_{\ell}(D)$, $m_h(D)$ have symbols $\rho$ and $1-\rho$ with $\rho\in C_c^{\infty}(B_2)$ satisfying $\rho|_{\overline{B}_1}=1$. As $v_R\in H^1(\R^n)\hookrightarrow H^s(\R^n)$, Lemma~\ref{Lemma: degenerate and decay}, \ref{neumann derivative} implies
		\begin{equation}
			\label{eq: neumann derivative of vR}
			-\lim_{y\to 0}y^{1-2s}\partial_y \mathcal{P}_{\sigma}^s v_R=d_s (-\Div(\sigma\nabla))^sv_R\in \dot{H}^{-s}(\R^n),
		\end{equation}
		where we used the property \eqref{eq: ellipticity of op} for the kernel of $(-\Div(\sigma\nabla))^s$ as well as the Gagliardo--Slobodeckij characterization of $\dot{H}^s(\R^n)$. Next, we assert that the same holds for the function $v_h$. To see this it is enough to show that $v_h\in H^1(\R^n)$. Note that there holds
		\[
		\begin{split}
			\int_{\R^n} \left|(1-\rho)\widehat{u}_1\right|^2\,d\xi&\leq C\int_{\R^n\setminus \overline{B}_1}|\xi|^{-2}\left|\xi\widehat{u}_1\right|^2\,d\xi\leq C\left\|\nabla u_1\right\|_{L^2(\R^n)}^2,
		\end{split}
		\]
		and hence we have $v_h\in L^2(\R^n)$. By a simple calculation one also gets $\nabla v_h\in L^2(\R^n)$ and hence $v_h\in H^1(\R^n)$.
		
		Note that there holds
		\[
		m_{\ell}(D)u_1=\mathcal{F}^{-1}( \rho\widehat{u}_1)=\check{\rho}\ast u_1
		\]
		with $\check{\rho}\in \schwartz(\R^n)$. It is immediate from Young's inequality that $m_h(D)u_1\in \dot{H}^1(\R^n)$. Hence, the same holds for $v_h$ and moreover $v_h=0$ on $\overline{B}_R$. 
		Finally, we can rely on \cite[Lemma 6.4]{CGRU2023reduction}, which is indeed a direct consequence of \eqref{eq: pointwise estimate y derivative}, to see that
		\[
		-\lim_{y\to 0}y^{1-2s}\partial_y \mathcal{P}_{\sigma}^s v_\ell\in L^2(\Omega).
		\]
		Let us discuss the solvability of \eqref{adjoint u_2}. By the trace theorem $\dot{H}^1(\R^{n+1}_+,y^{1-2s})\hookrightarrow \dot{H}^s(\R^n)$, the linear functional 
		\begin{equation}
			\label{eq: boundedness functional}
			\dot H^1(\R^{n+1}_+,y^{1-2s}) \ni \phi \mapsto \displaystyle \int_{\Omega}\phi(\cdot,0) \, \lim_{y\to 0}y^{1-2s}\mathcal{P}^s_\sigma u_1 \, dx
		\end{equation}
		is bounded, if it is interpreted accordingly (see \eqref{P u_1 in H-s}). As above, via the direct method in the calculus of variations, the problem \eqref{adjoint u_2} admits a solution $\wt u_2 \in \dot H^1(\R^{n+1}_+,y^{1-2s})$.
		
		More precisely, one can introduce the lower semicontinuous, coercive, convex energy functional
		\begin{equation}\label{energy functional}
			\begin{split}
				\mathcal{E}( \phi)  &\vcentcolon= \frac{1}{2}\int_{\R^{n+1}_+} y^{1-2s}\Sigma \nabla_{x,y}\phi \cdot \nabla_{x,y}\phi\, dxdy +\frac{d_s}{2}\int_{\Omega\times \{0\}} q \left| \phi\right|^2 \, dx \\
				&\quad \  +\int_{\Omega\times \{0\}} \phi\, \left(y^{1-2s}\mathcal{P}^s_\sigma u_1 \right)\, dx,
			\end{split}
		\end{equation}
		for $\phi \in \dot H^1_0(\R^{n+1}_+,y^{1-2s})$, and deduce that there exists a minimizer of \eqref{energy functional}, which is denoted by $\wt u_2\in \dot H^1(\R^{n+1}_+,y^{1-2s})$. The aforementioned properties of $\mathcal{E}$ can be seen similarly as for the simpler functional $E$ given by \eqref{energy functional 1}. In fact, the convexity is again clear and the first integral represents an equivalent norm on $\dot{H}^1(\R^{n+1}_+,y^{1-2s})$ and $0\leq q\in L^\infty(\Omega)$. Next, we assert that the embedding $\dot{H}^1(\R^{n+1},y^{1-2s})\hookrightarrow L^2(\Omega)$ is compact. Indeed, this follows by the chain of embeddings
		\[
		\dot{H}^1(\R^{n+1},y^{1-2s})\hookrightarrow \dot{H}^s(\R^n)\hookrightarrow H^s(\Omega)\hookrightarrow L^2(\Omega),
		\]
		where the first continuous embedding is the trace theorem for the weighted space $\dot{H}^1(\R^{n+1},y^{1-2s})$, the second the restriction to $\Omega$ and the last the compact Rellich--Kondrachov theorem. Thus, by using the equivalent norm induced from the first term in the definition of $\mathcal{E}$, \eqref{eq: boundedness functional}, and this compactness assertion, we see that $\mathcal{E}$ is weakly lower semicontinuous. It is also easy to see by the above estimates and $q\geq 0$ that $\mathcal{E}$ is coercive and hence as $H^1_0(\R^{n+1}_+,y^{1-2s})$ is weakly closed, we can again apply \cite[Theorem~1.2]{Variational-Methods} to conclude the existence of a minimizer $\widetilde{u}_2\in \dot{H}^1_0(\R^{n+1}_+,y^{1-2s})$ of $\mathcal{E}$.
		
		As \eqref{adjoint u_2} are the Euler--Lagrange equations of this minimization problem, we have obtained a solution of it, which is $\wt u_2$. Moreover, by non-negativity of $q$, \eqref{P u_1 in H-s} and \eqref{eq: boundedness functional}, $\widetilde{u}_2$ satisfies the energy estimate 
		\begin{equation}
			\left\| \wt u_2 \right\|_{\dot H^1(\R^{n+1}_+,y^{1-2s})}\leq C \left\| \lim_{y\to 0}y^{1-2s}\p_y \mathcal{P}^s_\sigma u_1 \right\|_{H^{-s}(\Omega)}<\infty,
		\end{equation}
		for some constant $C>0$ independent of $\wt u_2$. \\

		Now, consider the function 
		$$
		w\vcentcolon=\wt u_1 -\mathcal{P}^s_\sigma u_1 +\wt u_2, 
		$$ 
		then $w\in\dot H^1_{\mathrm{loc},0}(\overline{\R^{n+1}_+}, y^{1-2s})$ and $w$ solves \eqref{adjoint s in 0,1}. Furthermore, there holds
		\begin{equation}
			\label{eq: boundary cond of w}
			w(x,0)= u_1(x)-\mathcal{P}^s_\sigma u_1(x,0)+\wt u_2(x,0)=\wt u_2(x,0) \text{ for }x\in \R^n,
		\end{equation}
		which shows that indeed $w$ vanishes on $\Omega_e\times \{0\}$.
		Using that $\widetilde{u}_1=u_1$ is independent of $y$, this ensures that 
		\begin{equation}
			\label{eq: neumann der of w}
			\begin{split}
				\displaystyle -\lim_{y\to 0} y^{1-2s}\p_y w &= \lim_{y\to 0}y^{1-2s}\p_y \mathcal{P}^s_\sigma u_1 - \lim_{y\to 0}y^{1-2s}\p_y\wt u_2 \\
				&=-d_sq \wt u_2(x,0)=-d_s q w(x,0) \text{ for }x\in \Omega.
			\end{split}
		\end{equation}
		More concretely, for any $\psi \in H^1_{c,0}(\overline{\R^{n+1}_+}, y^{1-2s})$, there holds that
		\begin{equation}
			\begin{split}
				\displaystyle\int_{\R^{n+1}_+}y^{1-2s}\Sigma \nabla_{x,y}\mathcal{P}^s_\sigma u_1\cdot \nabla_{x,y}\psi \, dxdy &=-\int_{\Omega}\psi(x,0) \lim_{y\to 0}y^{1-2s}\p_y \mathcal{P}^s_\sigma  u_1\, dx, \\
				\int_{\R^{n+1}_+}y^{1-2s}\Sigma \nabla_{x,y}\wt  u_2 \cdot \nabla_{x,y}\psi \, dxdy &=-\int_{\Omega}\psi(x,0) \lim_{y\to 0}y^{1-2s}\p_y \mathcal{P}^s_\sigma u_1\, dx \\
				&\quad \, -d_s\int_{\Omega}\psi(x,0) q\widetilde{u}_2 \,  dx.
			\end{split}
		\end{equation}
		In addition, one has
		\begin{equation}
			\begin{split}
				\int_{\R^{n+1}_+} y^{1-2s}\Sigma \nabla_{x,y} \wt  u_1 \cdot \nabla_{x,y}\psi \, dxdy&=\int_{\R^{n+1}_+}y^{1-2s} \sigma \nabla u_1 \cdot \nabla \psi \, dxdy\\
				&=\int_{\R^n}\sigma \nabla u_1  \cdot \nabla  \left( \int_0^\infty y^{1-2s}\psi \,dy\right)\, dx \\
				&= -\left\langle  \varphi , \int_0^\infty y^{1-2s}\psi \, dy\right\rangle_{H^{-1}(\Omega)\times H^1(\Omega)},
			\end{split}
		\end{equation}
		where we used that by $\psi\in H^1_{c,0}(\overline{\R^{n+1}_+},y^{1-2s})$, Fubini's theorem and Jensen's inequality in particular there holds $\int_0^\infty y^{1-2s}\psi(x,y) \, dy\in H^1_c(\R^n)$. Using the weak definition of the Neumann derivative (see \eqref{eq: weak def of Neumann derivative}), we deduce from the above computations that there holds
		\[
		\begin{split}
			&\quad \, -\int_{\Omega}\psi(x,0) \lim_{y\to 0}y^{1-2s}\p_y w\, dx \\
			&=\int_{\R^{n+1}_+}y^{1-2s}\Sigma \nabla_{x,y}w\cdot \nabla_{x,y}\psi\, dxdy + \left\langle \varphi, \int_0^\infty y^{1-2s} \psi(\cdot,y)\, dy \right\rangle_{\wt H^{-1}(\Omega)\times H^1(\Omega)}\\
			&=-d_s\int_{\Omega}\psi(x,0) q(x)\widetilde{u}_2(x,0) \,  dx
		\end{split}
		\]
		for any $\psi\in H^1_{c,0}(\overline{\R^{n+1}_+}, y^{1-2s})$. Taking into account \eqref{eq: boundary cond of w}, this establishes \eqref{eq: neumann der of w}, and we complete the proof.
	\end{proof}

	Now, we are ready to prove the asserted Runge approximation. 

	\begin{proof}[Proof of Proposition \ref{prop: density}]
		First, we show the Runge approximation \eqref{equ: Runge error} by a contradiction argument. For this purpose let us suppose \eqref{equ: Runge error} does not hold, i.e., there exists $v\in S$ and $\eps>0$ such that 
		\begin{equation}\label{contradiction rigorous inequality}
			\left\| U_f - v \right\|_{H^1(\Omega)}\geq \eps >0, \text{ for all }f\in C^\infty_c(W).
		\end{equation}
		Arguing as in the formal proof above, it is enough to show the following implication:
		\begin{equation}\label{contradiction rigorous}
			\begin{split}
				\varphi( U_f ) =0 \text{ for all }f\in C^\infty_c(W) \implies \varphi(v)=0.
			\end{split}
		\end{equation}
		Let us point out that the rigorous version of the above formal proof, and its generalization to $s\in (0,1)$ as well as variable $\sigma$, is almost the same as the proof of \cite[Proposition 3.1]{CGRU2023reduction}, where one needs to consider the same vertical and transversal cutoff functions and use suitable decay properties of $\mathcal{P}^s_\sigma u_f$ as $y\to \infty$ (see Lemma \ref{Lemma: degenerate and decay}). To this end, consider $\eta_k(y)\vcentcolon=\eta_1(y/k)$, where $\eta_1 \in C^\infty_c([0,2])$ is a smooth cutoff function with $\eta_1=1$ near $y=0$ and $\int_0^\infty y^{1-2s} \eta_1 \, dy =1$.  By a change of variables it follows that 
        \[
            k^{2s-2} \int_0^\infty y^{1-2s}\eta_k \, dy =\int_0^\infty y^{1-2s} \eta_1 \, dy =1,
        \] 
        for all $k\in \N$.
		Without loss of generality, we may assume that there exists a large $R>0$ such that $\overline{\Omega}\cup \overline{W}\subset B_R$. In addition, let us introduce the sequence $\zeta_k(x):=\zeta_1(x/k)$, where $\zeta_1 \in C^\infty_c (B_{2R})$ is a smooth radial cutoff function with $\zeta_1\equiv 1$ in $\overline{B}_R$.

		Taking these smooth cutoff functions into account, let 
		\[
		( \mathcal{P}^s_\sigma u_{f})_k (x,y)\vcentcolon = \zeta_k(x)\eta_k(y)\mathcal{P}^s_\sigma u_f(x,y)
		\] 
		for $k\in \N$, which belongs to $H^1_{c}(\overline{\R^{n+1}_+}, y^{1-2s})$ as by Lemma~\ref{Lemma: degenerate and decay} we have $\mathcal{P}^s_{\sigma}u_f\in L^2_{loc}(\overline{\R^{n+1}_+},y^{1-2s})\cap \dot{H}^1(\R^{n+1}_+,y^{1-2s})$. By \eqref{contradiction rigorous} and Remark~\ref{Remark integration by parts}, we get
		\begin{equation}\label{duality 1}
			\begin{split}
				0&=\varphi ( U_f ) = \lim_{k\to \infty}\left\langle \varphi, \int_0^\infty y^{1-2s}( \mathcal{P}^s_\sigma u_{f})_k\, dy \right\rangle_{\wt H^{-1}(\Omega) \times H^1(\Omega)} \\
				&= \lim_{k\to \infty} \left(-\int_{\R^n} ( \mathcal{P}^s_\sigma u_{f})_k (\cdot,0)\lim_{y\to 0}y^{1-2s}\p_y w\, dx \right. \\
				&\qquad \qquad  \left.- \int_{\R^{n+1}_+} y^{1-2s}\Sigma \nabla_{x,y}w\cdot \nabla_{x,y}( \mathcal{P}^s_\sigma u_{f})_k \, dxdy\right) \\
				&=- \int_W f ( \lim_{y\to 0} y^{1-2s}\p_y w ) dx - \int_{\Omega} u_f ( \lim_{y\to 0} y^{1-2s}\p_y w ) dx + \lim_{k\to \infty}I_k,
			\end{split}
		\end{equation}
		where we used $\mathcal{P}_\sigma^s u_f(x,0)=u_f(x)$ for $x\in \R^n$, $\eta_k(0)=1$ and $u_f|_{\Omega_e}=f\in C_c^{\infty}(W)$. Furthermore, we set
		\begin{equation}\label{I_k}
			\begin{split}
				I_k \vcentcolon&=-\int_{\R^{n+1}_+}y^{1-2s}\Sigma ( \eta_k \zeta_k \nabla_{x,y} w) \cdot \nabla_{x,y }\mathcal{P}^s_\sigma u_f\, dxdy \\
				&\quad \, -\int_{\R^{n+1}_+}y^{1-2s} \Sigma \mathcal{P}^s_\sigma u_f \nabla_{x,y}( \eta_k \zeta_k ) \cdot \nabla_{x,y}w \, dxdy.
			\end{split}
		\end{equation}
		Moreover, in \eqref{duality 1}, we used that one has
		\begin{equation}
			\lim_{k\to\infty}\int_0^\infty y^{1-2s} \zeta_k(x)\eta_k(y)\mathcal{P}^s_\sigma u_f (x,y)\, dy=\int_0^\infty y^{1-2s} \mathcal{P}^s_\sigma u_f (x,y)\, dy \text{ in } H^1(\Omega).
		\end{equation}
		We next  claim that there holds
		\begin{equation}\label{claim in rigorous}
			-\int_{\Omega} u_f ( \lim_{y\to 0} y^{1-2s}\p_y w ) dx +\displaystyle\lim_{k\to \infty}I_k=0.
		\end{equation}
		Using the product rule for the first term in \eqref{I_k} and for the second term an integration by parts, we deduce
		\begin{equation}
			\begin{split}
				I_k &=-\int_{\R^{n+1}_+} y^{1-2s} \Sigma \nabla_{x,y} ( \eta_k \zeta_k w ) \cdot \nabla_{x,y} \mathcal{P}^s_\sigma u_f \, dxdy \\
				&\quad \, + \int_{\R^{n+1}_+} y^{1-2s}w \Sigma \nabla_{x,y}( \eta_k \zeta_k ) \cdot \nabla_{x,y}\mathcal{P}^s_\sigma u_f\, dxdy \\
				&\quad \, -\int_{\R^{n+1}_+}y^{1-2s} \Sigma \mathcal{P}^s_\sigma u_f \nabla_{x,y}( \eta_k \zeta_k ) \cdot \nabla_{x,y}w \, dxdy\\
				&= -\int_{\R^{n+1}_+} y^{1-2s} \Sigma \nabla_{x,y} ( \eta_k \zeta_k w ) \cdot \nabla_{x,y} \mathcal{P}^s_\sigma u_f \, dxdy \\
				&\quad \, + \int_{\R^{n+1}_+} y^{1-2s}w \Sigma \nabla_{x,y}( \eta_k \zeta_k ) \cdot \nabla_{x,y}\mathcal{P}^s_\sigma u_f\, dxdy \\
				&\quad \,  +\int_{\R^{n+1}_+} w \Div_{x,y} ( y^{1-2s}\Sigma \mathcal{P}^s_\sigma u_f \nabla_{x,y}( \eta_k\zeta_k) ) dxdy \\
				&= \int_{\R^n}\zeta_k w(x,0)\lim_{y\to 0}y^{1-2s}\p_y \mathcal{P}^s_\sigma u_f \, dx \\
				&\quad \, +2 \int_{\R^{n+1}_+} y^{1-2s}w \Sigma \nabla_{x,y}( \eta_k \zeta_k ) \cdot \nabla_{x,y}\mathcal{P}^s_\sigma u_f\, dxdy \\
				&\quad \, + \int_{\R^{n+1}_+} wy^{1-2s}\mathcal{P}^s_\sigma u_f \Div_{x,y}( \Sigma \nabla_{x,y} ( \eta_k \zeta_k) ) dxdy \\
				&\quad \, +(1-2s)\int_{\R^{n+1}_+}wy^{-2s}\mathcal{P}^s_\sigma u_f \zeta_k \p_y \eta_k \, dxdy.
			\end{split}
		\end{equation}
		In the second equality, we used and integration by parts and $\eta_k(y)=1$ for small $y$ and all $k\in \N$. In the last equality, we used the product rule for the term in the sixth line and for the term in the fourth line again an integration by parts together with $\eta_k(0)=1$ for $k\in\N$ as well as the PDE for $\mathcal{P}_{\sigma}^s u_f$. 
		
		Therefore, by the support conditions for the cutoff functions $\eta_k,\zeta_k$, we can write
		\begin{equation}\label{rigorous 1}
			\begin{split}
				&\quad \, -\int_{\Omega} u_f ( \lim_{y\to 0} y^{1-2s}\p_y w ) dx +\displaystyle\lim_{k\to \infty}I_k \\
				&=\lim_{k \to \infty}\int_{B_{2Rk}}\int_{0}^{2k} wy^{1-2s} \bigg\{ 2\Sigma\nabla_{x,y}( \zeta_k\eta_k ) \cdot \nabla_{x,y} \mathcal{P}^s_\sigma u_f  \\
				&\qquad \qquad \qquad  \qquad  +(  \Div_{x,y}( \Sigma \nabla_{x,y} ( \eta_k \zeta_k) )+ \frac{1-2s}{y}\zeta_k (\p_y \eta_k ) \mathcal{P}^s_\sigma u_f \bigg\} \, dydx.
			\end{split}
		\end{equation}
		More concretely, we have utilized that the boundary terms on $\R^n\times \{0\}$ are well-defined, and the equation for the adjoint equation \eqref{adjoint s in 0,1} implies that 
		\begin{equation}
			\begin{split}
				&\quad \,-\int_{\Omega} u_f ( \lim_{y\to 0} y^{1-2s}\p_y w ) dx +\underbrace{\lim_{k\to \infty}\int_{\R^n}\zeta_k w(x,0)\lim_{y\to 0}y^{1-2s}\p_y \mathcal{P}^s_\sigma u_f \, dx}_{w=0 \text{ on }\Omega_e\times \{0\}\text{ and }\zeta_k\to 1 \text{ as }k\to\infty}\\
				&=-\int_{\Omega} u_f \underbrace{ ( \lim_{y\to 0} y^{1-2s}\p_y w ) }_{=d_s qw } dx +\int_{\Omega\times \{0\}} w \underbrace{\lim_{y\to 0}y^{1-2s}\p_y \mathcal{P}^s_\sigma u_f}_{=-d_s(-\Div(\sigma\nabla))^su_f=d_squ_f} dx\\
				&=0,
			\end{split}
		\end{equation}
		which implies the identity \eqref{rigorous 1}.

		In addition, the rest of the proof is to ensure the right hand side of \eqref{rigorous 1} equals to zero. One can follow the same arguments as in the rigorous proof of \cite[Proposition 3.1]{CGRU2023reduction} to obtain
		\begin{equation}
			\begin{split}
				\left| I_k \right|  \leq  C (  I_{1,k}(w) + I_{2,k} (w)),
			\end{split}
		\end{equation}
		for some constant $C>0$ independent of $\mathcal{P}^s_\sigma u_f$, $w$ and $k$, where 
		\begin{equation}
			\begin{split}
				I_{1,k}(w)&\vcentcolon= k^{-1}\int_{\R^n}\int_{k}^{2k}y^{1-2s}|w|  \left( \left| \nabla_{x,y}\mathcal{P}^s_\sigma u_f \right| + \frac{\left| \mathcal{P}^s_\sigma u_f\right|}{k} \right) dydx, \\
				I_{2,k}(w)&\vcentcolon =k^{-1}\int_{B_{2Rk}\setminus B_{Rk}} \int_0^{2k}y^{1-2s}|w|  \left( \left| \nabla_{x,y}\mathcal{P}^s_\sigma u_f \right| + \frac{\left| \mathcal{P}^s_\sigma u_f\right|}{k} \right) dydx,
			\end{split}
		\end{equation}
		for $k\in \N$. Recalling that the function $w$ is constructed via $w=\wt u_1-\mathcal{P}^s_\sigma u_1 +\wt u_2$, and they have the same regularity properties as the ones in \cite{CGRU2023reduction}.
		Therefore, the rest of the argument is the same as in the rigorous proof of \cite[Proposition 3.1]{CGRU2023reduction}, so let us skip this lengthy analysis and for more details we refer the reader to the aforementioned article. As a result, there must hold $\lim_{k \to \infty}I_{1,k}=\lim_{k\to \infty}I_{2,k}=0$. Hence, we deduce that 
		\begin{equation}
			\int_W f ( \lim_{y\to 0} y^{1-2s}\p_y w ) dx=0, \text{ for all }f\in C^\infty_c(W).
		\end{equation}
		Since $f\in C^\infty_c(W)$ is arbitrary, we infer $\lim_{y\to 0}y^{1-2s}\p_y w =0$ in $W\times \{0\}$. Then the UCP for \eqref{adjoint s in 0,1} in the exterior domain implies that $w\equiv 0 $ in $\Omega_e\times (0,\infty)$ as desired.

		Now, for getting the implication \eqref{contradiction rigorous}, let us utilize the same cutoff functions $\beta_k(y)$ with $\beta_k(0)=0$ for $k\in\N$ (as constructed in Step 2 of the proof of \cite[Proposition 3.1]{CGRU2023reduction} or the proof of \cite[Proposition 6.1]{LLU2023calder}) to avoid boundary contributions on $\Omega\times \{0\}$. We can conclude by repeating the same argument as in the rest of the proof of \cite[Proposition 3.1]{CGRU2023reduction} and so the implication \eqref{contradiction rigorous} holds. For an outline of this argument, we refer the reader to the formal proof with $\sigma=1$ and $s=1/2$.
		
		Now, as explained in the formal proof, the conclusion $\varphi(v)=0$ contradicts the existence of a real number $\alpha$ satisfying \eqref{equ: contradiction} and hence we may deduce that such a function $v\in S$ having the property \eqref{contradiction rigorous inequality} cannot exist. Therefore, the Runge approximation holds eventually.
	\end{proof}

	\begin{remark}
		Notice that all above arguments hold even when $\sigma$ is an anisotropic, uniformly elliptic coefficient. We only need the isotropy of $\sigma$ to derive our uniqueness results and hence restrict ourselves to this case. Let us emphasize that the condition $\sigma=1$ in $\Omega_e$ is needed in order to derive $\lim_{k\to \infty}I_k=0$ in the rigorous proof of Proposition \ref{prop: density}, which is completely the same situation as in \cite{CGRU2023reduction}.
	\end{remark}
	
	\section{Proof of main results}\label{sec: proof of main thm}
	
	In this section, we show our main results.
	
	\subsection{Proof of Theorem \ref{thm: uniqueness of potential}}
	
	\begin{proof}[Proof of Theorem \ref{thm: uniqueness of potential}]
		Fix some functions $g,h\in H^{1/2}(\partial\Omega)$ and denote by $v\in H^1(\Omega)$ the unique solution to
		\[
		\begin{cases}
			\Div(\sigma\nabla v)=0&\text{ in }\Omega,\\
			v=g&\text{ on }\partial\Omega.
		\end{cases}
		\]
		By Proposition \ref{prop: density}, we can choose $( g_k)_{k\in\N}\subset C_c^{\infty}(W)$ such that $U_{g_k}\to v$ in $H^1(\Omega)$ as $k\to\infty$. Using the continuity of $\Lambda_{\sigma}\colon H^{1/2}(\partial\Omega)\to H^{-1/2}(\partial\Omega)$, the trace theorem and $\sigma|_{\partial\Omega}=1$ we may compute 
		\begin{equation}
			\label{eq: imp identity proof i}
			\begin{split}
				\left\langle \Lambda_\sigma g,h \right\rangle&=\lim_{k\to\infty}\left\langle \Lambda_{\sigma}( U_{g_k}|_{\partial\Omega}),h\right\rangle=\lim_{k\to\infty}\left\langle \partial_{\nu} v_k,h \right\rangle,
			\end{split}
		\end{equation}
		where $v_k\in H^1(\Omega)$ is the unique solution of 
		\begin{equation}
			\label{eq: PDE for vk}
			\begin{cases}
				\Div( \sigma\nabla v_k)=0&\text{ in }\Omega,\\
				v_k=\left.U_{g_k}\right|_{\partial\Omega}&\text{ on }\partial\Omega.
			\end{cases}
		\end{equation}
		In \eqref{eq: imp identity proof i} and for the rest of this article, we set $\langle\cdot,\cdot\rangle=\langle \cdot,\cdot\rangle_{H^{-1/2}(\partial\Omega)\times H^{1/2}(\partial\Omega)}$.

		Next, we want to see the identity
		\[
		\lim_{k\to\infty} \partial_\nu( U_{g_k}-v_k) =0\text{ in }H^{-1/2}(\partial\Omega).
		\]
		Note that by \eqref{reduced equation} and  \eqref{eq: PDE for vk} the function $w_k=U_{g_k}-v_k$ solves 
		\[
		\begin{cases}
			\Div( \sigma\nabla w_k) =d_s q u_{g_k}&\text{ in }\Omega,\\
			w_k=0&\text{ on }\partial\Omega.
		\end{cases}
		\]
		By writing $w_k= ( U_{g_k}-v) -( v_k-v)$, we have 
		\[
		w_k\to 0\text{ in }H^1(\Omega).
		\]
		In fact, $U_{g_k}-v\to 0$ in $H^1(\Omega)$ follows from Proposition~\ref{prop: density}. On the other hand, by the standard continuity estimate for elliptic equations (cf.~e.g.~ \cite[Chapter~8]{gilbarg2015elliptic}), we have
		\begin{equation}
			\label{eq: cont estimate for uniqueness proof i}
			\left\|v-v_k \right\|_{H^1(\Omega)}\leq C\left\|g- \left.U_{g_k}\right|_{\partial\Omega}\right\|_{H^{1/2}(\partial\Omega)},
		\end{equation}
		where we used that $v|_{\partial\Omega}=g$. Now, as $U_{g_k}|_{\partial\Omega}\to g$ in $H^{1/2}(\partial\Omega)$, we see from \eqref{eq: cont estimate for uniqueness proof i} that $v-v_k\to 0$ in $H^1(\Omega)$ as $k\to\infty$.
		Therefore
		\[
		\Div( \sigma\nabla w_k) =d_s qu_{g_k}\to 0\text{ in }H^{-1}(\Omega)\text{ as }k\to\infty.
		\]
		In fact, for any $\varphi\in H^1_0(\Omega)$ we have
		\begin{equation}
			\label{eq: limit potential term proof i}
			-d_s\lim_{k\to\infty} \int_{\Omega} q u_{g_k}\varphi\,dx=\lim_{k\to\infty}\int_{\Omega}\sigma \nabla w_k\cdot\nabla \varphi\,dx=0.
		\end{equation}
		If we multiply above PDE by $\varphi\in H^1(\Omega)$, then divergence theorem gives\footnote{Note that after rearranging terms, this is nothing else than the related (weak) normal derivative.}
		\[
		\begin{split}
			d_s\int_{\Omega}q u_{g_k}\varphi\,dx &=\int_{\Omega} \Div( \sigma\nabla w_k) \varphi\,dx \\
			&=\int_{\partial \Omega}( \partial_\nu w_k) \varphi|_{\partial\Omega}\,d\mathcal{H}^{n-1}-\int_{\Omega}\sigma\nabla w_k\cdot \nabla \varphi\,dx,
		\end{split}
		\]
		where the boundary integral has to be interpreted as a duality pairing between $H^{-1/2}(\partial\Omega)$ and $H^{1/2}(\partial\Omega)$.
		This shows
		\begin{equation}
			\label{eq: integration by parts formula limit proof i}
			\lim_{k\to\infty}\left|\int_{\partial \Omega} ( \partial_\nu w_k) \varphi|_{\partial\Omega}\,d\mathcal{H}^{n-1}-d_s\int_{\Omega}q u_{g_k}\varphi\,dx\right|=0.
		\end{equation}
		Now, let $\chi\in C_c^{\infty}(\Omega)$ be a cutoff function with $\chi=1$ on $\supp q$. Then for any $\varphi\in H^1(\Omega)$, using \eqref{eq: limit potential term proof i}, we may compute
		\begin{equation}
			\label{eq: potential term vanishes proof i}
			\begin{split}
				\lim_{k\to\infty}\int_{\Omega}qu_{g_k}\varphi\,dx&=\lim_{k\to\infty}\int_{\Omega}qu_{g_k}\chi \varphi\,dx=0.
			\end{split}
		\end{equation}
		Hence, by  \eqref{eq: integration by parts formula limit proof i} and \eqref{eq: potential term vanishes proof i} we get
		\[
		\begin{split}
			&\quad \, \lim_{k\to\infty}\left|\int_{\partial \Omega}( \partial_\nu w_k) \varphi|_{\partial\Omega}\,d\mathcal{H}^{n-1}\right|\\
			&\leq \limsup_{k\to\infty}\left|\int_{\partial \Omega}( \partial_\nu w_k)\varphi|_{\partial\Omega}\,d\mathcal{H}^{n-1}-d_s\int_{\Omega}q u_{g_k}\varphi\,dx\right| \\
			&\quad \, +\limsup_{k\to\infty}\left|d_s\int_{\Omega}q u_{g_k}\varphi\,dx\right|\\
			&=0.
		\end{split}
		\]
		This means that $\partial_{\nu} w_k\to 0$ in $H^{-1/2}(\partial\Omega)$ and therefore we get
		\[
		\lim_{k\to\infty} \partial_\nu U_{g_k}=\lim_{k\to\infty}\partial_\nu ( U_{g_k}-v_k) +\lim_{k\to\infty}\partial_\nu  v_k=\lim_{k\to\infty}\partial_\nu v_k\text{ in }H^{-1/2}(\partial\Omega).
		\]
		This ensures
		\begin{equation}
			\label{eq: identity for local DN map proof i}
			\begin{split}
				\left\langle \Lambda_\sigma g,h \right\rangle&=\lim_{k\to\infty}\left\langle \Lambda_{\sigma}( U_{g_k}|_{\partial\Omega}),h\right\rangle=\lim_{k\to\infty}\left\langle \partial_{\nu} v_k,h \right\rangle=\lim_{k\to\infty}\left\langle \partial_\nu U_{g_k},h \right\rangle
			\end{split}
		\end{equation}
		for any $g,h\in H^{1/2}(\partial\Omega)$. 
	\end{proof}

	\subsection{Proof of Theorem \ref{thm: uniqueness}}
	
	\begin{proof}[Proof of Theorem \ref{thm: uniqueness}]

		If $\mathcal{N}=\overline{\Omega}$, then $\sigma_1=\sigma_2$ in $\Omega$ and the uniqueness result follows from \cite[Theorem 1.1]{GLX}. Hence, we can assume without loss of generality that $\mathcal{N}$ is a proper subset of $\overline{\Omega}$. Let $u_{f_\ell}^{(j)}\in H^s(\R^n)$ be the solution to
		\begin{equation}
			\label{eq: opt tom eq thm 1.3}
			\begin{cases}
				( ( -\Div ( \sigma_j \nabla ) )^s +q_j)  u_{f_\ell}^{(j)} =0 &\text{ in }\Omega, \\
				u_{f_\ell}^{(j)} =f_\ell &\text{ in }\Omega_e,
			\end{cases}
		\end{equation}
		for $j,\ell \in \{1,2\}$, where $f_1,f_2 \in C^\infty_c(W)$ are arbitrarily smooth functions. By Lemma~\ref{Lem: C-S reduction}, we know that the functions
		\begin{equation}\label{eq: CS type extension main proof}
			U_{f_{\ell}}^{(j)}=\int_0^{\infty}y^{1-2s}\mathcal{P}_{\sigma_j}^s u_{f_{\ell}}^{(j)}\,   dy\in H^1_{\mathrm{loc}}(\R^n)
		\end{equation}
		solve
		\begin{equation}
			\label{eq: local-nonlocal PDE main proof}
			-\Div (  \sigma_j \nabla U_{f_{\ell}}^{(j)} ) =d_s ( -\Div (  \sigma_j \nabla ) )^s u_{f_{\ell}}^{(j)} \text{ in }\R^n
		\end{equation}
		for $j,\ell \in  \{1,2\} $, and by Lemma~\ref{lem: regularity} the equation \eqref{eq: local-nonlocal PDE main proof} holds in $H^{-s}(\R^n)$. Furthermore, note that $\mathcal{P}_{\sigma_j}^s u_{f_{\ell}}^{(j)}$ solves the degenerate elliptic equation 
		\begin{equation}
			\begin{cases}
				\Div_{x,y}( y^{1-2s}\Sigma_j \nabla_{x,y} 
				( \mathcal{P}_{\sigma_j}^s u_{f_{\ell}}^{(j)} )  ) =0 &\text{ in }\R^{n+1}_+, \\
				( \mathcal{P}_{\sigma_j}^s u_{f_{\ell}}^{(j)}) (x,0)=u_{f_{\ell}}^{(j)}(x) \text{ on }\R^n,
			\end{cases}
		\end{equation}
		for $j,\ell\in \{1,2\}$. 
		The above equation is derived from the Caffarelli-Silvestre type extension, and it has nothing to do with the nonlocal equation for $u_{f_{\ell}}^{(j)}\in H^s(\R^n)$, for $\ell=1,2$.

		In particular, the condition \eqref{nonlocal DN map same} implies
		\begin{equation}
			\begin{split}
				-\lim_{y\to 0}y^{1-2s} \p_y \mathcal{P}_{\sigma_1}^s u_{f_{\ell }}^{(1)} &=d_s( -\Div ( \sigma_1 \nabla ) )^su_{f_{\ell}}^{(1)}  \\
				&=d_s( -\Div ( \sigma_2 \nabla ) )^su_{f_{\ell}}^{(2)}  = -\lim_{y\to 0} y^{1-2s} \p_y \mathcal{P}_{\sigma_2}^s u_{f_{\ell}}^{(2)} \text{ in }W,
			\end{split}
		\end{equation}
		whenever $f_\ell \in C^\infty_c(W)$.
		By the assumption $\sigma_1=\sigma_2$ in the open neighborhood $\mathcal{N}\subset \overline{\Omega}$ of $\p \Omega$, one knows $\sigma\vcentcolon =\sigma_1 =\sigma_2$ in $\mathcal{N}\cup \Omega_e$. In particular, for any $f_\ell \in C^\infty_c(W)$ the difference $V=\mathcal{P}_{\sigma_1}^s u_{f_{\ell}}^{(1)}-\mathcal{P}_{\sigma_2}^s u_{f_{\ell}}^{(2)}$ satisfies 
		\begin{equation}
			\label{eq: UCP argument}
			\begin{cases}
				\Div_{x,y}( y^{1-2s} \Sigma\nabla_{x,y} 
				V ) =0 &\text{ in }( \mathcal{N}\cup \Omega_e ) \times (0,\infty) , \\
				V = \displaystyle \lim_{y\to 0}y^{1-2s}\p_y V =0 &\text{ on }W\times \{0\}.
			\end{cases}
		\end{equation}
		Then by the UCP for the PDE in \eqref{eq: UCP argument}, one can conclude that  
		\begin{equation}\label{extension u_1=u_2}
			\mathcal{P}_{\sigma_1}^s u_{f_{\ell}}^{(1)}=\mathcal{P}_{\sigma_2}^s u_{f_{\ell}}^{(2)} \text{ in }( \mathcal{N}\cup \Omega_e ) \times (0,\infty),
		\end{equation} 
		for any $f_\ell \in C^\infty_c(W)$ and $\ell=1,2$. To obtain this one can directly invoke the results in \cite[Section 5]{GLX} or argue as follows. First of all, by the condition $\sigma|_{\Omega_e}=1$, there holds
		\[
		\begin{cases}
			\Div_{x,y}( y^{1-2s}\nabla_{x,y} 
			V ) =0 &\text{ in }( \mathcal{N}\cup \Omega_e ) \times (0,\infty) , \\
			V = \displaystyle \lim_{y\to 0}y^{1-2s}\p_y V =0 &\text{ on }W\times \{0\}.
		\end{cases}
		\]
		Secondly, after an even reflection of $V$ which requires that the normal derivative of $V$ vanishes on $W\times \{0\}$, we deduce from \cite[Theorem~2.3.12]{fabes1982local} that $V$ is locally H\"older continuous on $W\times [0,\infty)$. Therefore, we can apply \cite[Proposition~2.2]{ruland2015unique} to see that $V=0$ on $B_r((x,0))\cap \overline{\R^{n+1}_+}$ for some $x\in W$ and $r>0$, where $B_r((x,0))$ is the ball in $\R^{n+1}$ with radius $r>0$ and center $(x,0)$ such that $B_r((x,0))\cap \{y=0\}\subset W$. But now we can apply the usual UCP for the differential operator in \eqref{eq: UCP argument} on the sets of the form $(\mathcal{N}\cup \Omega_e)\times (y_0,y_1)$ for any sufficiently small $y_0>0$ and $y_1>y_0$ and may conclude that $V=0$ on these sets. Thus, in the end we get $V=0$ in $(\mathcal{N}\cup\Omega_e)\times (0,\infty)$ as claimed.
		
		Thus, via the definition of the function $U_{f_\ell}^{(j)}$ for $j,\ell \in \{1,2\} $, then there holds that 
		\begin{equation}\label{U_1=U_2}
			\begin{split}
				U_{f_\ell}^{(1)}=\int_0^\infty y^{1-2s}\mathcal{P}^{s}_{\sigma_1}u_{f_\ell}^{(1)}(\cdot,y)\, dy=\int_0^\infty y^{1-2s}\mathcal{P}^{s}_{\sigma_2}u_{f_\ell}^{(2)}(\cdot,y)\, dy=U_{f_\ell}^{(2)}
			\end{split}
		\end{equation}
		in $\mathcal{N}\cup \Omega_e$, which will be used in the forthcoming proof.

		We next show 
		\begin{equation}\label{q_1=q_2 in nbhd of boundary}
			q_1=q_2 \text{ in }\mathcal{N}\cap \Omega.
		\end{equation}
		Combining the condition \eqref{nonlocal DN map same} and \eqref{DN-integral id 2} in Lemma~\ref{lemma: diff DN maps}, we obtain 
		\begin{equation}
			B_{\sigma_1,q_1}  ( u^{(1)}_{f_1},u^{(2)}_{f_2}) - B_{\sigma_2,q_2}  ( u^{(2)}_{f_1},u^{(2)}_{f_2})  =0,
		\end{equation}
		which by the definition of the bilinear forms (see \eqref{eq: bilinear form}) is equivalent to 
		\begin{equation}\label{integral id 1}
			\begin{split}
				&\left\langle( -\Div(\sigma_1 \nabla))^s u^{(1)}_{f_1} -( -\Div (\sigma_2 \nabla)) ^s u^{(2)}_{f_1}, u^{(2)}_{f_2}\right\rangle_{H^{-s}(\R^n)\times H^s(\R^n)}\\ 
				& \quad + \int_{\Omega}	( q_1u^{(1)}_{f_1} -q_2u^{(2)}_{f_1})  u^{(2)}_{f_2}\, dx=0.
			\end{split}
		\end{equation}
		First inserting \eqref{eq: local-nonlocal PDE main proof} into \eqref{integral id 1} and then decomposing $u_{f_2}^{(2)}=( u_{f_2}^{(2)}-f_2) +f_2$, we obtain by \eqref{U_1=U_2} the identity
		\begin{equation}
			\label{eq: identity step 1}
			\begin{split}
				&-\frac{1}{d_s}\left\langle \Div( \sigma_1 \nabla U^{(1)}_{f_1}) -\Div ( \sigma_2 \nabla U^{(2)}_{f_1}) , u^{(2)}_{f_2}-f_2\right\rangle_{H^{-s}(\R^n)\times H^s(\R^n)}\\ 
				& \quad + \int_{\Omega}	(  q_1 u^{(1)}_{f_1} -q_2u^{(2)}_{f_1} )  ( u^{(2)}_{f_2}-f_2 )  dx=0.
			\end{split}
		\end{equation}
		Recall that by Lemma~\ref{lem: regularity}, we have
		\begin{equation}
			\label{eq: divergence estimate}
			\begin{split}
				&\quad \, \left\langle \Div( \sigma_1 \nabla U^{(1)}_{f_1}) -\Div ( \sigma_2 \nabla U^{(2)}_{f_1}) , \varphi\right\rangle_{H^{-s}(\R^n)\times H^s(\R^n)}\\
				&=\left\langle \Div( \sigma_1 \nabla U^{(1)}_{f_1}) -\Div ( \sigma_2 \nabla U^{(2)}_{f_1}) , \varphi\right\rangle_{H^{-s}(\Omega)\times \widetilde{H}^s(\Omega)}\\
				&\leq \left( \left\|\Div( \sigma_1 \nabla U^{(1)}_{f_1})\right\|_{H^{-s}(\Omega)}+ \left\|\Div ( \sigma_2 \nabla U^{(2)}_{f_1})\right\|_{H^{-s}(\Omega)}\right)\|\varphi\|_{\widetilde{H}^s(\Omega)}
			\end{split}
		\end{equation}
		for all $\varphi\in \widetilde{H}^s(\Omega)$. We next want to derive useful integral identities in order to show the uniqueness of the potentials.

		Let us fix some $\varphi\in C_c^{\infty}(\mathcal{N}\cap \Omega)$. By the Runge approximation (Proposition \ref{prop:(Runge-approximation-property)}), there exists a sequence of exterior data $\left\{f_{2,m}\right\}_{m\in \N}\subset C^\infty_c(W)$ such that 
		\begin{equation}
			u_{f_{2,m}}^{(2)}-f_{2,m} \to \varphi \text{ in }\wt H^s(\Omega) \text{ as }m\to \infty.
		\end{equation}
		With this sequence of functions at hand, we obtain 
		\begin{equation}\label{eq: identity step 1'}
			\begin{split}
				&\quad \, \lim_{m\to\infty}\, \left\langle \Div( \sigma_1 \nabla U^{(1)}_{f_{1}}) -\Div ( \sigma_2 \nabla U^{(2)}_{f_{1}}) , u^{(2)}_{f_{2,m}}-f_{2,m}\right\rangle_{H^{-s}(\R^n)\times H^s(\R^n)}\\ 
				&=  \left\langle \Div( \sigma_1 \nabla U^{(1)}_{f_{1}}) -\Div ( \sigma_2 \nabla U^{(2)}_{f_{1}}) , \varphi\right\rangle_{H^{-s}(\R^n)\times H^s(\R^n)}\\ 
				& =\left\langle \Div( \sigma_1 \nabla U^{(1)}_{f_{1}}) -\Div ( \sigma_2 \nabla U^{(2)}_{f_{1}}) , \varphi\right\rangle_{\distr(\R^n)\times \test(\R^n)}\\ 
				& =\int_{\R^n} \left[ U^{(1)}_{f_{1}}\Div( \sigma_1\nabla \varphi) -U^{(2)}_{f_{1}}\Div( \sigma_2\nabla \varphi) \right] \,dx\\
				&=\int_{\mathcal{N}\cap \Omega}\left[ U^{(1)}_{f_{1}}\Div( \sigma_1\nabla \varphi) -U^{(2)}_{f_{1}}\Div( \sigma_2\nabla \varphi)\right]\,dx\\
				& =0,
			\end{split}
		\end{equation}
		where we used that $\supp(\varphi)\subset \mathcal{N}\cap \Omega$ and the last integral vanishes by the fact that $\sigma_1 =\sigma_2$ in $\mathcal{N}$ and \eqref{U_1=U_2}.
		Therefore, by passing to the limit $m\to\infty$ in \eqref{eq: identity step 1} (with $f_2=f_{2,m}$) and using \eqref{eq: identity step 1'} we get
		\[
		\int_{\mathcal{N}\cap \Omega}(  q_1 u_{f_1}^{(1)}-q_2 u_{f_1}^{(2)}) \varphi\,dx=0
		\]
		for all $\varphi\in C_c^{\infty}(\mathcal{N}\cap \Omega)$. Hence, we can conclude that 
		\begin{equation}\label{q1u1=q2u2 in nbd of boundary}
			q_1u^{(1)}_{f_1}=q_2u^{(2)}_{f_1} \text{ in } \mathcal{N}\cap \Omega.
		\end{equation}
		Moreover, by using \eqref{extension u_1=u_2}, we also have 
		\begin{equation}
			u_{f_1}^{(1)}(x)=	( \mathcal{P}_{\sigma_1}^s u_{f_{1}}^{(1)} ) (x,0)=( \mathcal{P}_{\sigma_2}^s u_{f_{1}}^{(2)}) (x,0)=u_{f_1}^{(2)}(x), \text{ for }x\in \mathcal{N}.
		\end{equation} 
		This implies
		\[
		(  q_1-q_2)  u^{(1)}_{f_1}=0 \text{ in }\mathcal{N}\cap \Omega,
		\]
		for any $f_1\in C_c^{\infty}(W)$. Fix some nonzero $f_1\in C_c^{\infty}(W)$. Now, suppose that there exists $x\in \mathcal{N}\cap \Omega$ such that $q_1(x)\neq q_2(x)$. By continuity there exists an open ball $B_r(x)\subset \Omega\cap \mathcal{N}$ such that $q_1(y)\neq q_2(y)$ for all $y\in B_r(x)$ and thus one has $u^{(1)}_{f_1}=0$ on this set. By \eqref{eq: opt tom eq thm 1.3}, one then obtains $(-\Div(\sigma_1\nabla))^s u_{f_1}^{(1)}=0$ on $B_r(x)$. The UCP together with the conditions 
		\[
		u_{f_1}^{(1)}=(-\Div(\sigma_1\nabla))^s u_{f_1}^{(1)}=0\text{ on }B_r(x)
		\]
		imply that $u_{f_1}^{(1)}=0$ in $\R^n$. This in turn contradicts the assumption $f_1\neq 0$. Thus, we must have $q_1=q_2$ in $\mathcal{N}\cap\Omega$.
	\end{proof}		
	
	\subsection{Proof of Theorem \ref{thm: unique det}}
	
	Now, taking Proposition~\ref{prop: density} into account, we can turn the local uniqueness result (Theorem \ref{thm: uniqueness}) in certain cases into a global uniqueness result. Before presenting the proof, we need to introduce some preliminary material. 
	
	For any $1<p<\infty$ and any bounded Lipschitz domain $V\subset\R^n$, we denote by $E_p(V)$ the Banach space of vector fields $F\in L^p(V;\R^n)$ such that $\Div F\in L^p(V)$ and its natural norm is given by
	\begin{equation}
		\|F\|_{E_p(V)}\vcentcolon = \|F\|_{L^p(V;\R^n)}+\|\Div F\|_{L^p(V)}.
	\end{equation}
	It is well-known that smooth vector fields $F\in C^{\infty}(\overline{V};\R^n)$ are dense in $E_p(V)$ (see \cite[eq.~(1.2.16)]{sohr2012navier}). Furthermore, throughout this section $p'$ stands for the H\"older conjugate exponent of $1<p<\infty$.
	
	The importance of the space $E_p(V)$ comes from the fact that vector fields in that space have a well-defined normal trace.
	
	\begin{lemma}[{\cite[Lemma~1.2.2--1.2.3]{sohr2012navier}}]
		\label{lemma: normal trace}
		Let $1<p<\infty$ and suppose $V\subset \R^n$ is a bounded Lipschitz domain. Then there exists a bounded linear operator 
		$$
		\Gamma_N^V\colon E_p(V)\to (W^{1-1/p,p}(\partial V))^*,
		$$
		which is called \emph{generalized normal trace} and is also denoted occasionally by $\nu\cdot F|_{\partial V}$ with $\nu$ being the normal outward pointing vector field to $\partial V$. The normal trace $\Gamma_N^V$ has the following properties:
		\begin{enumerate}[(i)]
			\item\label{prop 1 norm trace} If $F\in C^{\infty}(\overline{V},\R^n)$ and $\varphi\in W^{1-1/p',p'}(\partial V)$, then one has
			\begin{equation}
				\label{eq: normal trace smooth case}
				\left\langle \Gamma_N^V F,\varphi\right\rangle=\int_{\partial V}(\nu \cdot F)\varphi\,d\mathcal{H}^{n-1}.
			\end{equation}
			\item\label{prop 2 norm trace} If $F\in E_p(V)$ and $\Phi\in W^{1,p'}(\partial V)$, then there holds
			\begin{equation}
				\label{eq: integration by parts formula for E}
				\left\langle \Gamma_N^V F,\Phi|_{\partial V}\right\rangle =\langle \Phi,\Div F\rangle_{L^2(V)}+ \left\langle\nabla \Phi,F \right\rangle_{L^2(V)}.
			\end{equation}
		\end{enumerate}
	\end{lemma}
	
	With these results at hand, we can prove the following auxiliary lemma.
	
	\begin{lemma}
		\label{auxialiary lemma}
		Let $1<p<\infty$, suppose $\Omega\Subset \Omega'\subset\R^n$ are bounded Lipschitz domains and set $\widetilde{\Omega}=\Omega'\setminus \overline{\Omega}$. Then for any $F\in E_p(\Omega')$ and $\varphi\in W^{1-1/p'}(\partial\Omega)$, there exists $\Phi\in W^{1,p'}(\Omega')$ such that $\Phi|_{\partial\Omega}=\varphi$, $\Phi|_{\partial\Omega'}=0$ and
		\begin{equation}
			\label{eq: integration by parts formula intermediate bdry}
			\left\langle \Gamma_N^{\Omega}F|_{\Omega},\varphi \right\rangle=-\left(\langle \Phi,\Div F\rangle_{L^2(\widetilde{\Omega})}+\langle \nabla \Phi,F\rangle_{L^2(\widetilde{\Omega})}\right).
		\end{equation}
	\end{lemma}
	
	\begin{proof}
		As $F\in E_p(\Omega')$, we can choose $F_k\in C^{\infty}(\overline{\Omega'};\R^n)$ such that $F_k\to F$ in $E_p(\Omega')$ as $k\to\infty$. On the other hand, for any given $\varphi\in W^{1-1/p',p'}(\partial\Omega)$ we can select $\Psi \in W^{1,p'}(\Omega)$ such that $\Psi|_{\partial\Omega}=\varphi$. Next, we may extend $\Psi$ to a function $\Phi\in W^{1,p'}(\Omega')$ such that $\Phi|_{\partial\Omega'}=0$ (see~e.g.~\cite[Section~5.4]{EvansPDE}). Now, using Lemma~\ref{lemma: normal trace} we may compute 
		\[
		\begin{split}
			\left\langle \Gamma_N^{\Omega}F,\varphi \right\rangle&=\lim_{k\to\infty}\left\langle \Gamma_N^{\Omega}F_k,\varphi\right\rangle\\
			&=\lim_{k\to\infty}\int_{\partial\Omega}(F_k\cdot\nu)\varphi\,d\mathcal{H}^{n-1}\\
			&=-\lim_{k\to\infty}\int_{\partial\widetilde{\Omega}}(F_k\cdot n)\Phi|_{\partial\widetilde{\Omega}}\,d\mathcal{H}^{n-1}\\
			&=-\lim_{k\to\infty}\left\langle \Gamma_N^{\widetilde{\Omega}}F_k,\Phi|_{\partial\widetilde{\Omega}}\right\rangle\\
			&=-\lim_{k\to\infty}\left(\left\langle \Phi,\Div F_k\right\rangle_{L^2(\widetilde{\Omega})}+\left\langle\nabla \Phi,F_k\right\rangle_{L^2(\widetilde{\Omega})}\right)\\
			&=-\left(\langle \Phi,\Div F\rangle_{L^2(\widetilde{\Omega})}+\langle\nabla \Phi,F\rangle_{L^2(\widetilde{\Omega})}\right),
		\end{split}
		\]
		where $\nu$ denotes the outwards pointing vector field to $\partial\Omega$ and $n$ the outwards pointing normal vector field to $\partial\widetilde{\Omega}$.
	\end{proof}
	
	\begin{proof}[Proof of Theorem \ref{thm: unique det}]
		
		Again by \cite[Theorem 1.1]{GLX} we can assume without loss of generality that $\mathcal{N}$ is a proper subset of $\overline{\Omega}$. Let $\sigma_1,\sigma_2$ and $q_1,q_2$ be given as in the statement and let us assume that $q_2=0$ in $\Omega$. Fix some $f\in C_c^{\infty}(W)$ and denote by $u_f^{(j)}\in H^s(\R^n)$ the unique solution of 
		\begin{equation}
			\label{eq: nonlocal opt tom eq main thm 1}
			\begin{cases}
				( ( -\Div ( \sigma_j \nabla ) )^s +q_j)  u_{f}^{(j)} =0 &\text{ in }\Omega, \\
				u_{f}^{(j)} =f &\text{ in }\Omega_e,
			\end{cases}
		\end{equation}
		for $j=1,2$.
		Next we claim that for any $f\in C_c^{\infty}(W)$ the nonlocal Cauchy data 
		$(f|_{W},   ( - \Div( \sigma_j \nabla) ) ^s u_f^{(j)} |_{W})$
		determines $(U_f^{(j)}|_{\p \Omega}, \p_\nu U_f^{(j)}|_{\p \Omega})$ for $j=1,2$. Here, as usual, the functions $U_f^{(j)}$ are given by \eqref{eq: CS type extension main proof}, which solve \eqref{eq: local-nonlocal PDE main proof} with $f_\ell=f$,	for $j=1,2 $.

		By the proof of Theorem \ref{thm: uniqueness}, we already know that there holds
		\begin{equation}\label{U_1=U_2 1}
			\begin{split}
				u_f^{(1)}&=u_f^{(2)}, \quad  U_{f}^{(1)}=U_{f}^{(2)}\text{ in }\Omega_e\cup \mathcal{N}\text{ and } q_1=q_2=0\text{ in }\mathcal{N}\cap \Omega, 
			\end{split}
		\end{equation}
		for any $f\in C_c^{\infty}(W)$. Thus, the trace theorem gives
		\begin{equation}
			\label{eq: conclusion trace theorem}
			U_f^{(1)} |_{\partial\Omega}= U_f^{(2)} |_{\partial\Omega}\text{ in }H^{1/2}(\partial\Omega).
		\end{equation}
		Next, we want to show that also the normal derivatives coincide, that is
		\begin{equation}
			\label{eq: equ neumann der}
			\partial_\nu U_f^{(1)}|_{\partial\Omega}=\partial_\nu U_f^{(2)}|_{\partial\Omega}\text{ in }H^{-1/2}(\partial\Omega)
		\end{equation}
		for any  $f\in C_c^{\infty}(W)$. For this purpose let us introduce the vector field 
		\[
		F\vcentcolon = \sigma_1\nabla U_f^{(1)}-\sigma_2\nabla U_f^{(2)}\in L^2_{\mathrm{loc}}(\R^n).
		\]
		Moreover, let us choose a bounded smooth domain $\Omega'\subset\R^n$ such that $\Omega\Subset \Omega'$. Now, we show that $\Div F\in L^2(\Omega')$. First select some open set $V\subset\R^n$ satisfying $\Omega\setminus\mathcal{N}\Subset V\Subset \Omega$ and a cutoff function $\rho\in C_c^{\infty}(\Omega)$ such that $\rho=1$ in $V$. Now, for any $\varphi\in C_c^{\infty}(\Omega')$ we may calculate
		\begin{equation}\label{some comp}
			\begin{split}
				\langle \Div F,\varphi\rangle&= -\int_{\Omega'}( \sigma_1 \nabla U_f^{(1)}-\sigma_2\nabla U_f^{(2)}) \cdot\nabla \varphi\,dx\\
				&=\underbrace{-\int_{V}( \sigma_1 \nabla U_f^{(1)}-\sigma_2\nabla U_f^{(2)}) \cdot\nabla \varphi\,dx}_{\text{$\sigma_1=\sigma_2$ on $\mathcal{N}\cup \Omega_e$ and \eqref{U_1=U_2 1}}}\\
				&=\underbrace{-\int_{V}( \sigma_1 \nabla U_f^{(1)}-\sigma_2\nabla U_f^{(2)}) \cdot\nabla (\rho\varphi)\,dx}_{\rho =1 \text{ on }V}\\
				&=\underbrace{-\int_{\Omega}( \sigma_1 \nabla U_f^{(1)}-\sigma_2\nabla U_f^{(2)}) \cdot\nabla (\rho\varphi)\,dx}_{\text{$\sigma_1\nabla U_f^{(1)}=\sigma_2\nabla U_f^{(2)}$ on $\mathcal{N}$ (see \eqref{U_1=U_2 1})}}\\
				&=\underbrace{d_s\int_{\Omega}( q_1u_f^{(1)}-q_2u_f^{(2)}) \rho\varphi\,dx}_{\text{$\rho\varphi\in C_c^{\infty}(\Omega)$}}\\
				&=d_s\int_{\Omega'}( q_1u_f^{(1)}-q_2u_f^{(2)}) \rho\varphi\,dx.
			\end{split}
		\end{equation}
		Thus, 
		\begin{equation}
			\label{eq: divergence of F}
			\Div F=d_s( q_1u_f^{(1)}-q_2u_f^{(2)}) \rho\in L^2(\Omega')
		\end{equation}
		as we wanted to see. This in turn implies $F\in E_2(\Omega')$. By Lemma~\ref{auxialiary lemma}, we get
		\[
		\left\langle \Gamma_N^{\Omega}F|_{\Omega},\varphi \right\rangle=-\left(\langle \Phi,\Div F\rangle_{L^2(\widetilde{\Omega})}+\langle \nabla \Phi,F\rangle_{L^2(\widetilde{\Omega})}\right)
		\]
        for any $\varphi\in H^{1/2}(\partial\Omega)$ and $\Phi\in H^1(\Omega')$ with $\Phi|_{\partial\Omega}=\varphi$ and $\Phi|_{\partial\Omega'}=0$.
		By \eqref{U_1=U_2 1} and \eqref{eq: divergence of F}, we know that $F=\Div F=0$ a.e. in $\widetilde{\Omega}$ and hence
		\[
		\left\langle \Gamma_N^{\Omega}F|_{\Omega},\varphi \right\rangle=0
		\]
		for any $\varphi\in H^{1/2}(\partial\Omega)$, which yields the assertion \eqref{eq: equ neumann der}.
		
		Now, by the proof of Theorem \ref{thm: uniqueness of potential}, for any $g\in H^{1/2}(\partial\Omega)$ and corresponding solution $v_1\in H^1(\Omega)$ of
		\begin{equation}
			\begin{cases}
				\Div ( \sigma_1 \nabla v_1 )  =0 &\text{ in }\Omega, \\
				v_1 =g &\text{ on }\p \Omega,
			\end{cases}
		\end{equation}
		there exists $g_k\in C^\infty_c(W)$ such that  $U_{g_k}^{(1)}\to v_1$ in $H^1(\Omega)$ as $k\to\infty$ and $U_{g_k}^{(1)}|_{\partial\Omega}\to g$ in $H^{1/2}(\partial\Omega)$. Furthermore, we know that there holds
		\[
		\left\langle \Lambda_{\sigma_1} g,h \right\rangle=\lim_{k\to\infty}\langle \partial_\nu U_{g_k}^{(1)},h\rangle
		\]
		for any $h\in H^{1/2}(\partial\Omega)$. By  \eqref{eq: equ neumann der} this implies
		\[
		\left\langle \Lambda_{\sigma_1} g,h \right\rangle=\lim_{k\to\infty}\left\langle \partial_\nu U_{g_k}^{(1)},h\right\rangle=\lim_{k\to\infty}\left\langle \partial_\nu U_{g_k}^{(2)},h\right\rangle.
		\]
		Since $q_2=0$, we deduce from Lemma~\ref{Lem: C-S reduction} that $U_{g_k}^{(2)}\in H^1(\Omega)$ solves
		\[
		\begin{cases}
			\Div( \sigma_2 \nabla v)=0 &\text{ in }\Omega, \\
			v =\left. U_{g_k}^{(2)} \right|_{\partial\Omega} &\text{ on }\p \Omega.
		\end{cases}
		\]
		This ensures that there holds
		\[
		\left\langle \Lambda_{\sigma_1} g,h \right\rangle=\lim_{k\to\infty}\left\langle \Lambda_{\sigma_2}( U_{g_k}^{(2)}|_{\partial\Omega}),h\right\rangle.
		\]
		Finally, the continuity of the DN map $\Lambda_{\sigma_2}$ and $ U_{g_k}^{(2)} |_{\partial\Omega}=U_{g_k}^{(1)}|_{\partial\Omega}\to g$ in $H^{1/2}(\partial\Omega)$ as $k\to\infty$ gives
		\[
		\left\langle \Lambda_{\sigma_1} g,h \right\rangle=\left\langle \Lambda_{\sigma_2} g,h \right\rangle
		\]
		for all $g,h\in H^{1/2}(\partial\Omega)$. Hence, from \cite{sylvester1987global} we deduce that $\sigma_1=\sigma_2$ in $\Omega$ as we already know $\sigma_1=\sigma_2$ on $\partial\Omega\subset \mathcal{N}$. Finally, the condition \eqref{nonlocal DN map same} together with $\sigma_1=\sigma_2$ in $\R^n$ and \cite[Theorem 1.1]{GLX}, demonstrates that $q_1=q_2=0$ in $\Omega$ as desired. This proves the assertion.    
	\end{proof}
	
	\bigskip

	\noindent\textbf{Acknowledgments.} The authors would like to thank Mikko Salo and Ali Feizmohammadi for fruitful discussions and suggestions to improve this article.
	\begin{itemize}
		\item Y.-H. Lin is partially supported by the National Science and Technology Council (NSTC) Taiwan, under the project 113-2628-M-A49-003. Y.-H. Lin is also a Humboldt research fellow.  
		\item P.~Zimmermann was supported by the Swiss National Science Foundation (SNSF), under the grant number 214500.
	\end{itemize}
	
	\section*{Statements and Declarations}
	
	\subsection*{Data availability statement}
	No datasets were generated or analyzed during the current study.
	
	\subsection*{Conflict of Interests} Hereby we declare there are no conflict of interests.

	\bibliography{refs} 
	
	\bibliographystyle{alpha}
	
\end{document}